\documentclass[10pt]{article}

\usepackage{etex}
\usepackage[margin={1.2in}]{geometry}
\usepackage[titletoc]{appendix}

\usepackage{amsmath,blkarray,booktabs,bigstrut}

\usepackage{enumerate}

\usepackage{tikz,eso-pic}
\usepackage{tikz-3dplot}

\usepackage{tikz}
\usetikzlibrary{matrix}

\usepackage{eucal}
\usepackage{mathrsfs}
\usepackage{stmaryrd}

\usepackage{hyperref} 

\usepackage{rotating}

\usepackage{longtable}

\usepackage[sc]{mathpazo}
\usepackage{courier}

\usepackage[utf8]{inputenc}
\usepackage[T1]{fontenc}

\usepackage[english]{babel}
\usepackage{blindtext}

\usepackage{amsmath}

\usepackage{microtype}

\usepackage{amsmath}
\usepackage{amssymb}
\usepackage{amsthm}
\usepackage{color}
\usepackage{latexsym,extarrows}

\usepackage[all]{xy}
\usepackage{graphicx}
\usepackage{subcaption}

\usepackage[stable,multiple]{footmisc}

\RequirePackage[font=small,format=plain,labelfont=bf,textfont=it]{caption}
\makeatletter
\newsavebox{\@brx}
\newcommand{\llangle}[1][]{\savebox{\@brx}{\(\m@th{#1\langle}\)}%
  \mathopen{\copy\@brx\kern-0.5\wd\@brx\usebox{\@brx}}}
\newcommand{\rrangle}[1][]{\savebox{\@brx}{\(\m@th{#1\rangle}\)}%
  \mathclose{\copy\@brx\kern-0.5\wd\@brx\usebox{\@brx}}}
\makeatother

\newcommand\xqed[1]{%
	\leavevmode\unskip\penalty9999 \hbox{}\nobreak\hfill
	\quad\hbox{#1}}
\newcommand\xxqed{\xqed{$\triangle$}}

\def\e#1\e{\begin{equation}#1\end{equation}}
\def\ea#1\ea{\begin{align}#1\end{align}}

\theoremstyle{plain}
\newtheorem{thm}{Theorem}[section]
\newtheorem{lem}[thm]{Lemma}
\newtheorem{prop}[thm]{Proposition}

\theoremstyle{definition}
\newtheorem{dfn}[thm]{Definition}
\newtheorem{ex}[thm]{Example}
\newtheorem{rem}[thm]{Remark}

\newcommand{\op}{\operatorname}
\newcommand{\C}{\mathbb{C}}

\newcommand{\Z}{\mathbb{Z}}
\renewcommand{\H}{\mathbf{H}}

\newcommand{\im}{\op{im}}

\newcommand{\pii}{2\pi \mathbf{i}}

\def\Xint#1{\mathchoice
{\XXint\displaystyle\textstyle{#1}}%
{\XXint\textstyle\scriptstyle{#1}}%
{\XXint\scriptstyle\scriptscriptstyle{#1}}%
{\XXint\scriptscriptstyle\scriptscriptstyle{#1}}%
\!\int}
\def\XXint#1#2#3{{\setbox0=\hbox{$#1{#2#3}{\int}$}
\vcenter{\hbox{$#2#3$}}\kern-.5\wd0}}

\def\dashint{\Xint-}

\numberwithin{equation}{section}

\makeatletter
\newcommand{\subjclass}[2][2010]{%
  \let\@oldtitle\@title%
  \gdef\@title{\@oldtitle\footnotetext{#1 \emph{Mathematics Subject Classification.} #2}}%
}
\newcommand{\keywords}[1]{%
  \let\@@oldtitle\@title%
  \gdef\@title{\@@oldtitle\footnotetext{\emph{Key words and phrases.} #1.}}%
}
\makeatother

\makeatletter
\let\orig@afterheading\@afterheading
\def\@afterheading{%
   \@afterindenttrue
  \orig@afterheading}
\makeatother

\begin{document}
\title{\bf 
Gromov-Witten Generating Series of Elliptic Curves 
and Iterated Integrals of Eisenstein-Kronecker Forms
}
\author{Jie Zhou}
\date{}

\maketitle

\begin{abstract}
	
	We compute various types of iterated integrals 
	of Eisenstein-Kronecker forms that are constructed from the Kronecker theta function.
Furthermore, we relate the generating series of Gromov-Witten invariants of elliptic curves
to these iterated integrals.

\end{abstract}
\setcounter{tocdepth}{2} \tableofcontents

\section{Introduction}

The Gromov-Witten (GW) theory of elliptic curves has remained an interesting and active research topic in modern enumerative geometry.
Explicit formulas for the generating series of the corresponding GW invariants in terms of Jacobi theta functions have been derived by Okounkov-Pandharipande \cite{Okounkov:2006}  and Bloch-Okounkov \cite{Bloch:2000}, 
and turn out to be quasi-elliptic functions \cite{Eichler:1985} that are interesting in the studies of modular forms.
These formulas have since received a great deal of attention 
due to their concreteness and effectiveness in studying the interaction between enumerative geometry,  modular forms,  and mathematical physics, see e.g., \cite{Dijkgraaf:1995, Kaneko:1995, Zagier:2016partitions}.\\

Recently, it is shown 
in \cite{Zhou:2023GWgeneratingseries} that these purely enumerative generating series
are given by configuration space integrals
of  cohomological classes constructed from sections of  Poincar\'e bundles.
Furthermore, they admit new
sum-over-partitions formulas that are much simpler than the original ones given in \cite{Bloch:2000, Okounkov:2006}.

We now review and rephrase these results briefly.
For any  point $\tau$ on the upper half-plane $\H$, let
\begin{equation}
	E=\C/\Lambda_{\tau}\,,\quad \Lambda_{\tau}:=\Z\oplus \Z \tau
\end{equation}
be the corresponding elliptic curve.
We fix once and for all a
linear holomorphic coordinate $z$ on the universal cover $\widetilde{E}=\mathbb{C}\rightarrow E$. 
The origin on $E$ is taken to be the image of $0\in \mathbb{C}$ under the universal covering map $\mathbb{C}\rightarrow E$.
We also fix a weak Torelli marking $\{[A],[B]\}$, namely a symplectic frame for $H_{1}(E,\mathbb{Z})$, where the class $[A]$ is taken to be the image  under the map $\mathbb{C}\rightarrow E$ of the segment $A$ connecting $ \tau$ and $ \tau+1$ on $\mathbb{C}$.

For any $n\geq 1$, denote $[n]:=(1,2,\cdots, n)$ and let $\mathsf{Conf}_{[n]}(E)\subseteq E^{n}$ be the configuration space of $n$ ordered points on $E$.
Let the linear coordinates on the $n$ components of $\widetilde{E}^{n}, E^{n}$ in $\widetilde{E}^{n}\times E^{n}$ be $w_1,\cdots, w_{n}$ and $z_1,\cdots, z_{n}$ respectively, and $q=e^{\pii\tau}$.

\begin{thm}[=Theorem  \ref{mainthm1}]\label{mainthm1intro}

Assume the following notation.
\begin{itemize}
	\item 

	Let 
	\begin{equation}\label{eqnomegaC2nintro}
	\omega_{n}:=
	{\theta (\sum_{i=1}^{n} w_{i}) \over \prod_{i=1}^{n} \theta (w_{i})}\cdot
	{	\prod_{i<j }\theta (z_{i}+w_{i}-z_{j}-w_{j})\theta (z_{i}-z_{j}) \over  \prod_{i<j}
		\theta (z_{i}+w_{i}-z_{j}) \theta (z_{i}-w_{j}-z_{j})}\cdot \bigwedge_{i=1}^{n}{\mathbf{i} \,dz_{i}\wedge d\bar{z}_{i}\over 2\,\mathrm{im}\,\tau}\,,
\end{equation}
 where $\theta(w)$ is the odd Jacobi theta function  normalized such that $(\partial_w \theta)(0)=1$.
 \item
	Let $\widehat{\mathbold{B}}_{n}(w)$ be the 
	complete Bell polynomial in the variables $\widehat{\mathcal{E}}_{m}^{*}(w),m\geq 1$ which are the Eisenstein-Kronecker series\footnote{A quick definition of the complete Bell polynomial in the variables $x_1,\cdots, x_{m},\cdots$ is given by $\mathbold{B}_{m}(x_1,\cdots,x_{m})=m!\cdot [t^{m}]\exp(\sum_{k=1}^{\infty}{t^{k}\over k!}x_{k} )$. The quantity $\widehat{\mathcal{E}}_{m}^{*}$ can be alternatively written as 
	$\partial_{w}^{m}(\ln\theta)+2\mathbb{G}_{m}+\delta_{m,1}{-\pi\over \mathrm{im}\,\tau}(\overline{w}-w)$, where $2\mathbb{G}_{m}$ is the ordinary Eisenstein series defined as the summation of $(m-1)!\cdot\lambda^{-m}$ over $\lambda\in \Lambda_{\tau}\setminus\{(0,0)\}$.}
	\[
	\widehat{\mathcal{E}}_{m}^{*}(w):=
	(-1)^{m-1}(m-1)!
	\left({1\over w^m}+\sum_{\lambda\in \Lambda_{\tau}\setminus\{(0,0)\}}\left( {1\over (w+\lambda)^{m}}+{(-1)^{m-1}}{1\over \lambda^{m}}\right)  +\delta_{m,1}{-\pi\over \mathrm{im}\,\tau} (\overline{w}-w)
	\right)\,.
	\]
	\item 
	Denote $\Pi_{[n]\setminus \{j\}}$ to be the set of partitions of $[n]\setminus \{j\}$ consisting of 
	elements of the form $\pi=\{\pi_1,\cdots, \pi_{\ell}\}$.
\end{itemize}
	Then the regularized integral \cite{Li:2020regularized} $\dashint_{E^{n}}\omega_{n}$ satisfies
	\begin{equation}\label{eqnTngwpureweightstructureintro}
		\dashint_{E^{n}}\omega_{n}=
		\sum_{j\in [n]}\sum_{\pi\in \Pi_{[n]\setminus \{j\}}}
		{\widehat{\mathbold{B}}_{{\pi}}\over |\pi|}\,,
		\quad
		{\widehat{\mathbold{B}}_{{\pi}}\over |\pi|}:=\prod_{k=1}^{\ell} {\widehat{\mathbold{B}}_{|\pi_{k}|}(\sum_{i\in \pi_{k}}w_{i})\over |\pi_{k}|}\,.
	\end{equation}
 
\end{thm}

The significance of this result lies in the following. Let $F_n(w_{1},\cdots, w_{n}; q)$ be the $n$-point function of  disconnected stationary GW invariants  \cite{Okounkov:2006} of the elliptic curve $\mathcal{E}$ that is mirror to $E$. Consider its scaled version
	\begin{equation}\label{eqndefTn}
		T_{n}(w_{1},\cdots, w_{n};q):=\theta(\sum_{i=1}^{n} w_{i})\cdot \prod_{k=1}^{+\infty} (1-q^{k})^{-1}\cdot F_{n}(w_{1},\cdots, w_{n};q)\,.
	\end{equation}
This quantity admits very explicit expressions in terms of the Jacobi theta function $\theta$ and turns out to be a quasi-elliptic function,
due to Okounkov-Pandharipande \cite{Okounkov:2006}  and Bloch-Okounkov \cite{Bloch:2000}.
 Let $\widehat{T}_{n}(w_{1},\cdots, w_{n};q)$ be the elliptic and modular completion \cite{Eichler:1985} of $T_{n}(w_{1},\cdots, w_{n};q)$.
 Then one has the following relation \cite{Zhou:2023GWgeneratingseries}
between this enumerative-geometric quantity and the regularized integral 
  \begin{equation}\label{eqnTnhatasregularizedintegral}
 \widehat{T}_{n}(w_{1},\cdots, w_{n};q)
		=\dashint_{E^{n}}\omega_{n}\,.
 \end{equation}

The results \eqref{eqnTngwpureweightstructureintro}, \eqref{eqnTnhatasregularizedintegral} are  proved \cite[Theorem 1.1, Theorem 1.2]{Zhou:2023GWgeneratingseries}  
in their holomorphic limit versions, which correspond
to an iterated $A$-cycle integral on the configuration space $\mathsf{Conf}_{[n]}(E)$. 
The results are then promoted \cite[Section 4.3]{Zhou:2023GWgeneratingseries} to the present version using the machinery \cite{Eichler:1985} of quasi-elliptic functions,
and the relation between iterated $A$-cycle integrals and regularized integrals established in \cite{Li:2020regularized}.\\

The main goal of this work is to offer a direct and much simpler proof of 
Theorem \ref{mainthm1intro} using    the tool of regularized integrals
\cite{Li:2020regularized, Li:2022regularized, Zhou:2023cohomologicalpairings}
and properties of the Eisenstein-Kronecker forms (reviewed in Section \ref{secpreliminaries} below).
In particular, this work serves as
a demonstration of the usefulness of the notion of regularized integrals which seems to have provided 
an ideal tool in regularizing divergent integrals arising from 2d chiral conformal field theories (see \cite{Li:2022regularized} for discussions on this).
In the course, we compute various types of iterated 
integrals of Eisenstein-Kronecker forms which are of independent interest in understanding the geometry of configuration spaces of elliptic curves.
We also reveal a relation between regularized integrals of Eisenstein-Kronecker forms
and Chen's iterated path integrals  \cite{Chen:1977}, and obtain some combinatorial  results for the generating series of $\{\widehat{T}_{n}\}_{n\geq 1}$.

		\subsection*{Structure of the paper}

	In Section \ref{secpreliminaries}, we  recall the basics of quasi-elliptic and almost-elliptic functions, including
		Jacobi theta functions, Kronecker theta function, and 
	Eisenstein-Kronecker series.

	In Section 	\ref{seciteratedintegrals},  after a brief review 
		on the notion of regularized integrals on configuration spaces of elliptic curves,
		we compute  regularized integrals of Eisenstein-Kronecker forms and discuss their connections to Chen's iterated path integrals.

In Section 	\ref{secGW}, we prove Theorem \ref{mainthm1intro}
based on the results obtained in Section \ref{seciteratedintegrals}.
We also discuss some  combinatorial properties of the generating series of $\{\widehat{T}_{n}\}_{n\geq 1}$.

	\subsection*{Acknowledgements}
	
	We thank the anonymous referees for useful comments that helped improve the paper.
	 We thank Si Li for fruitful collaborations on related topics, and Jin Cao, Xinxing Tang, and Zijun Zhou for helpful conversations.
	 
This work was supported by the
	National Key Research and Development
	Program of China (No. 2022YFA1007100) and the Young Overseas High-Level Talents Introduction Plan of China.

\subsection*{Notation and conventions}
	\begin{itemize}
		\item Let $\Omega_{X}^{p}$ be the sheaf of holomorphic $p$-forms on $X$ and 
		$H^{0}(X,\Omega_{X}^{p})$ be the corresponding space of global sections over $X$.
		\item Let $A^{p,q}_{X}$ (resp. $A^{p,q}_{X}(\star D)$) be the space of $(p,q)$ forms on $X$ (resp. the space of $(p,q)$ forms that are 
		smooth on $X$ except for holomorphic poles along the effective divisor $D$).
		Let $A^{n}(X)=A^{n}_{X}=\bigoplus_{p+q=n}A^{p,q}_{X}, A^{n}_{X}(\star D)=\bigoplus_{p+q=n}A^{p,q}_{X}(\star D)$.

			\item 
		The notation $[t^k]f$ represents the degree $t^{k}$ coefficient of the formal Laurent series $f$ in $t$.
		

		\item
		Throughout this work, when we need to keep track of the variables in a construction such as a function $f_{n}$ that are taken from some finite set $J$ or sequence $\mathbold{K}$
		of variables  with cardinality $n$, we use the notation
		$f_{J}$ or $f_{\mathbold{K}}$.
		Correspondingly, we shall
		use the notation $S$ abusively for both the set $S$ and its cardinality $|S|$.	
	
		\item
		We denote collectively the points $P_{i},1\leq i \leq n$ and $Q_{j},1\leq j\leq n$ by $P,Q$, respectively.
		Similarly, we apply the same convention to denote a collection of coordinates, e.g., $z=(z_1,\cdots, z_n)$.
	
	\item
		For notational simplicity, we shall often 
		suppress the arguments of a function whenever they are clear from the surrounding texts.

	\end{itemize}

\section{Preliminaries}
\label{secpreliminaries}

We now 
collect
a few results in the literature about the space of almost-elliptic functions, which
includes in particular the subspace $A_{E}^{0}(\star D)$ that we are mostly interested in.

Throughout this work,
constructions on the elliptic curve $E_{\tau}=\C/\Lambda_{\tau}$ (and on the corresponding universal family
$\mathcal{E}_{\H}\rightarrow \H$) will be identified with their lifts
to the universal cover.
For example, meromorphic functions on $E$
are identified with those on $\mathbb{C}$
that are periodic/elliptic under the translation action by $\Lambda_{\tau}$.

\subsection{Jacobi theta functions and Eisenstein-Kronecker series}
\label{basicsonellipticfunctions}

We first recall some standard facts about Jacobi theta functions and 
Eisenstein series, mainly following \cite{Eichler:1985, Silverman:2009arithmetic}.
Define $\theta$ to be the normalized Jacobi theta function $\vartheta_{({1\over 2},{1\over 2})}$: 
\begin{equation}
\theta(z)={\vartheta_{({1\over 2},{1\over 2})}(z)\over \vartheta'_{({1\over 2},{1\over 2})}(0)}={\vartheta_{({1\over 2},{1\over 2})}(z)\over -2\pi \eta^3}\,,
\end{equation}
where $\eta$ is the Dedekind eta function.
For a positive integer $k$, define the Eisenstein series
\begin{equation}\label{eqnwidehatZexpansion}
	G_{k}={1\over 2}\sum_{\lambda\in\Lambda_{\tau}\setminus \{(0,0)\}}^{\quad\,e} {1\over \lambda^{k}}\,.
\end{equation}	
When $k$ is odd, it is defined to be zero by convention.
When $k=2$, one uses the Eisenstein summation prescription $\sum^{e}$ (summing over the $1$ direction first then $ \tau$ direction in $\Lambda_{\tau}$) to deal with the non-absolute convergence issue.

Let $\zeta,\wp$ be the Weierstrass $\zeta$-function and $\wp$-function,
with $\wp=-\partial_{z}\zeta=-\zeta'$. Hereafter the superscript $'$ indicates derivative in $z$.
The following relation between Eisenstein series and  semi-periods is classical
\begin{equation}
\int_{A}\wp \,dz=-\eta_{1}\,,\quad \eta_{1}:=2G_{2}\,.
\end{equation}
One also has 
the following expansions (see e.g., \cite[Chapter VI.3]{Silverman:2009arithmetic})
\begin{eqnarray}\label{eqnthetazetaexpansion}
		\theta(z)
	&=&z\prod_{\lambda\in\Lambda_{\tau}\setminus\{(0,0)\}}\left(1-{z\over \lambda}\right)e^{{z\over \lambda}+{1\over 2}({z\over \lambda})^2}=z\exp\left(\sum_{k\geq 1}-2G_{2k}{z^{2k}\over 2k}\right)
	\,,\nonumber\\ 
	\zeta(z)&=&
	{1\over z}+\sum_{ \lambda\in\Lambda_{\tau}\setminus\{(0,0)\}}\left( {1\over z+\lambda}-{1\over \lambda}+{z\over \lambda^2}\right)={1\over z}-\sum_{k\geq 2} {2G_{2k}}{z^{2k-1}}\,.
\end{eqnarray}	
For later use, we introduce the following notation.

\begin{dfn}\label{dfnhatthetazeta}
	Define
\begin{equation}
	\mathbold{Y}=-{\pi\over \mathrm{im}\,\tau}\,,\quad 
	\mathbold{A}=	(\bar{z}-z)\mathbold{Y}=\pii\,{\bar{z}-z\over \bar{\tau}-\tau}\,.
\end{equation}
	Introduce 	
	\begin{eqnarray}\label{eqnhatthetazeta}
		\widehat{\eta}_1&:=&
		\eta_{1}+\mathbold{Y}\,,\nonumber\\
		2\widehat{G}_{k}&:=&	2G_{k}+\delta_{k,2}\mathbold{Y}\,,\nonumber\\
		\widehat{\theta}(z)&:=&\theta(z)\cdot \exp(-2\pi {(\mathrm {im}~z)^2\over \mathrm{im}~\tau})\,,\nonumber\\
		Z(z)&:=&\partial_{z} \ln {\theta}=\zeta(z)-z{\eta}_{1}\,,\nonumber\\
		\widehat{Z}(z)&:=&\partial_{z} \ln \widehat{\theta}=\zeta(z)-z\widehat{\eta}_{1}		+\mathbold{Y}\bar{z}\,.
	\end{eqnarray}	
\end{dfn}
Hereafter, we suppress the notation for the variable $\bar{z}$
that is linearly independent of $z$.
However, when considering properties such as parity, a notation 
$f(-z)$
stands for the quantity obtained from replacing both $z,\bar{z}$ in 
$f$ by $-z,-\bar{z}$.\\

We also recall the Eisenstein-Kronecker series \cite{Weil:1976}.
\begin{dfn}\label{dfnEisensteinKroneckerseries}
The Eisenstein-Kronecker series are defined to be
\begin{equation}\label{eqndfnKroneckerEisenstein}
	\mathcal{E}_{m}(z)=\sum^{~\quad e}_{\lambda\in\Lambda_{\tau}} {1\over  ( z+\lambda)^{m}}\,,\quad m\geq 1\,.
\end{equation}
Introduce the
normalized versions
\begin{eqnarray}\label{eqnEm*}
	\mathcal{E}_{m}^{*}(z)&=&
	(-1)^{m-1}(m-1)! \,
	\mathcal{E}_{m}(z)
	+(m-1)!\cdot 2G_{m}\,,\nonumber\\
	\widehat{\mathcal{E}}_{m}^{*}(z)&=&\mathcal{E}_{m}^{*}(z)+\delta_{m,1}\mathbold{A}(z)
	\,.
\end{eqnarray}
\end{dfn}

Definition \ref{dfnEisensteinKroneckerseries} and the relation \eqref{eqnthetazetaexpansion}  give
\begin{equation}\label{eqnEm}
	\partial_{z}\mathcal{E}_{m}=-m\mathcal{E}_{m+1}\,,\quad 
	\mathcal{E}_{m}={(-1)^{m-1}\over (m-1)!}(\ln\theta)^{(m)}\,,\quad
	m\geq 1\,.
\end{equation}

\subsection{Laurent coefficients of the Kronecker theta function}

Of crucial importance in this work are the Kronecker theta function $	S_{c}(z)$ \cite{Weil:1976}  or Szego kernel, and their Laurent coefficients which will be used to define the Eisenstein-Kronecker forms in Section \ref{secEKforms} below.

\begin{dfn}\label{dfnSzegokernel}
Set
\begin{equation}
	S_{c}(z):={\theta(z+c)\over \theta(z)\theta(c)}\,,\quad 
	\widehat{S}_{c}(z):=
	e^{c\mathbold{A}(z)}{\theta(z+c)\over \theta(z)\theta(c)}\,.
\end{equation}
	These functions are meromorphic in $c$. Denote their Laurent expansions in $c$ by
\begin{equation}
	S_{c}(z)=
	{1\over c}\sum_{m\geq 0} c^{m}{{\mathbold{B}}_{m}(z)\over m!}\,,\quad 
		\widehat{S}_{c}(z)={1\over c}\sum_{m\geq 0} c^{m}{\widehat{\mathbold{B}}_{m}(z)\over m!}
	\,.
\end{equation}
\end{dfn}
\begin{rem}
The quantity $S_{c}(z)$ appears as the two-point function in chiral free fermions \cite{Raina:1989}, and 
$\widehat{S}_{c}(z)$ is its elliptic completion. 
The above Laurent coefficients enter the
elliptic polylogarithms \cite[Section 3.5]{Brown:2011} (see also \cite{Bannai:2007, Bannai:2010}), and the GW generating series of elliptic curves \cite[Section 4.3]{Zhou:2023GWgeneratingseries}.
See \cite{Sprang:2019, Sprang:2020} for a possible unified viewpoint via the Poincar\'e sheaf, and \cite{Fonseca:2020, Fonseca:2025} for a purely algebraic description of these real-analytic objects.
\end{rem}

\subsubsection{Ring of almost-elliptic functions}

The Laurent coefficients of the Kronecker theta function $S_{c}(z)$ can be computed concretely in terms of the Eisenstein-Kronecker series $\mathcal{E}_{m}(z),m\geq 1$ as follows.
One has from \eqref{eqnEm} that 
\[
\mathcal{E}_{1}(z+c)
=
\sum_{m\geq 0}
\partial_{z}^{m}\mathcal{E}_{1}(z) {c^{m}\over m!}
=\sum_{m\geq 0}(-1)^{m}m! \,
\mathcal{E}_{m+1}(z) {c^{m}\over m!}\,,
\]
and thus
\begin{equation}\label{eqngeneratingseriesofEms}
{\theta(z+c)\over \theta(z)}
=\exp
\left(\sum_{m\geq 0}
(-1)^{m}m! \,
\mathcal{E}_{m+1}(z){c^{m+1}\over (m+1)!}
\right)\,.
\end{equation}
Combining \eqref{eqnthetazetaexpansion} and \eqref{eqnEm*}, this gives 
\begin{equation}
S_{c}(z)={1\over c}\exp \left(\sum_{m\geq 1} {c^{m}\over m!} \mathcal{E}_{m}^{*}(z)\right)\,.
\end{equation}
Recall that 
the complete Bell polynomial in the variables $x_1,\cdots, x_{m},\cdots$ is defined as  \begin{equation}\label{eqndfncompleteBellpolynomial}
\mathbold{B}_{m}(x_1,\cdots,x_{m}):=m!\cdot [t^{m}]\exp\left(\sum_{k=1}^{\infty}{t^{k}\over k!}x_{k} \right)\,.
\end{equation}
By the set partition version of the Fa\`a di Bruno formula, one then sees that
$\mathbold{B}_{m}(z)$ in Definition \ref{dfnSzegokernel} is the 
complete Bell polynomial $\mathbold{B}_{m}(\mathcal{E}_{1}^{*},\mathcal{E}_{2}^{*},\cdots, \mathcal{E}_{m}^{*})$ in the variables 
$\mathcal{E}_{k}^{*}, k\geq 1$.
Similarly, $\widehat{\mathbold{B}}_{m}(z)$ is given by the complete Bell polynomial in the  variables 
$\widehat{\mathcal{E}}_{k}^{*}, k\geq 1$.
 More details regarding various expansions and formulas for $S_{c}(z), \widehat{S}_{c}(z)$ can be found in
\cite{Zagier:1991}
(see also \cite{Brown:2011, Goujard:2016counting, Zhou:2023GWgeneratingseries}).

\begin{dfn}
		\label{dfnringofquasiellipticfunctions}
	Let the notation be as above.
	Set
\begin{equation}
	e_{m}(z):=
{{\mathbold{B}}_{m}(z)\over m!}
\,,\quad
\widehat{e}_{m}(z):=
{\widehat{\mathbold{B}}_{m}(z)\over m!}\,,\quad m\geq 0\,,
\end{equation}
and use the convention $e_{m}(z)=0=\widehat{e}_{m}(z)$ if $m\leq -1$.
	Define the ring of \emph{quasi-elliptic} 
	functions $	\widetilde{\mathfrak{F}}$ 
	 to be the
	 graded $\mathbb{C}$-algebra\footnote{The quantities $\{{e}_{m}(z)\}_{m\geq 1}, \{\widehat{e}_{m}(z)\}_{m\geq 1}$ 
	satisfy relations such as ones induced by the Weierstrass relations. These relations, however, 
	do not play an important role in this work.}
generated by $\{e_{m}\}_{m\geq 1}$, and the ring of \emph{almost-elliptic} 	functions $	\widehat{\mathfrak{F}}$ to
be the
graded $\mathbb{C}$-algebra generated by $\{\widehat{e}_{m}\}_{m\geq 1}$.
	Here for $m\geq 1$, the gradings of $e_{m},\widehat{e}_{m} $ are $m$ and the polynomial degrees are $1$.
	
\end{dfn}

Using \eqref{eqnwidehatZexpansion}, it is direct to see that 
the first few terms are given by 
\begin{equation}\label{eqnfirstfewhatems}
	\widehat{e}_{0}(z)=1\,,\quad 
	\widehat{e}_{1}(z)=\widehat{Z}(z)\,,\quad
	\widehat{e}_{2}(z)={1\over 2}(-\wp(z)+\widehat{Z}^{2}(z))\,.
\end{equation}
From 
the binomial identity, we also have 
\begin{equation}\label{eqnbinomialidentity}
\widehat{e}_{m}(z)=\sum_{k=0}^{m}e_{k}(z){(\mathbold{A}(z))^{m-k}\over (m-k)!}\,.
\end{equation}
For example, from \eqref{eqnhatthetazeta} and  the automorphy 
\begin{equation}\label{eqnautomorphyoftheta}
\theta(z+1)=-\theta(z)\,,\quad 
\theta(z+\tau)=-e^{-\pi \mathbf{i}\tau}e^{-\pii z}\theta(z)\,,
\end{equation}
one has the quasi-ellipticity and almost-ellipticity behavior 
\begin{eqnarray}\label{eqne1automorphy}
	&&e_{1}(z+1)=e_{1}(z)\,,\quad 
e_{1}(z+ \tau)=e_{1}(z)-\pii\,,\nonumber\\
	&&\widehat{e}_{1}(z+1)=\widehat{e}_{1}(z)\,,\quad 
\widehat{e}_{1}(z+ \tau)=\widehat{e}_{1}(z)\,.
\end{eqnarray}

The generators $\{e_{m}(z)\}_{m\geq 1}, \{\widehat{e}_{m}(z)\}_{m\geq 1}$ are actually  quasi-Jacobi and almost-Jacobi forms of index zero with the weight given by the grading. See  \cite{Eichler:1985,  Libgober:2011, Goujard:2016counting}
for more details.

We shall frequently make use of the following relation without explicit mentioning
\[
\widehat{\mathfrak{F}}\subseteq  \widetilde{\mathfrak{F}}\otimes \mathbb{C}[\mathbold{A}]\,.
\]
The map $\widehat{\mathfrak{F}}\rightarrow\widetilde{\mathfrak{F}} $, obtained by taking the degree zero term  in the polynomial expansion in $\mathbold{A}$, offers an isomorphism of these two rings.
We call this map the ``holomorphic limit'' and denote it by $\lim_{\mathbold{A}=0}$;
its inverse map is called the ``elliptic completion''.

\begin{rem}\label{remchangeofgeneratorsforringofalmostellipticfunctions}

Since $\{\mathbold{B}_{m}\}_{m\geq 0},\{\widehat{\mathbold{B}}_{m}\}_{m\geq 0}$ are the complete Bell polynomials
in the generators $\{\mathcal{E}^{*}_{m}\}_{m\geq 1}$, $\{\widehat{\mathcal{E}}^{*}_{m}\}_{m\geq 1}$, the above 
rings can alternatively defined by taking the generators to be $\{\mathcal{E}^{*}_{m}\}_{m\geq 1},\{\widehat{\mathcal{E}}^{*}_{m}\}_{m\geq 1}$, respectively. 	
While the latter sets are more convenient in analyzing properties relate to non-meromorphicity (such as holomorphic anomaly \cite{Li:2022regularized}) due to the fact that $\{\mathcal{E}^{*}_{m}\}_{m\geq 2}$ are all meromorphic, the former sets can sometimes be more convenient for actual computations due to the simple pole structure of them (see Lemma 	\ref{lempolarparts} below),
 as we shall see in Section \ref{seciteratedintegrals}.

\end{rem}

\subsubsection{Singularities and differential relations of Laurent coefficients}

We recall some simple facts about the Laurent coefficients $\widehat{e}_{m}(z),m\geq 1$, mostly from \cite{Brown:2011}.
For self-containedness, we also give their short proofs. \\

Due to \eqref{eqne1automorphy} and  \cite[Equation A.10]{Zhou:2023GWgeneratingseries} (cf. \cite[Proposition 1.25]{Bannai:2010}),
${e}_{m},m\geq 2$ is smooth at $z\in \mathbb{Z}$,
while ${e}_{1}=Z$ has a simple pole  with residue $1$ at $z\in \mathbb{Z}$.
Note that since the ${e}_{m}$'s are not almost-elliptic, these functions do not descend to functions on $E$.

We have the following more precise statement regarding the singularities of $\widehat{e}_{m},m\geq 1$.

\begin{lem}	\label{lempolarparts}
			Let the notation be as above. 
	\begin{enumerate}[i).]
		\item 
		
		The functions 	$
		\widehat{e}_{m}(z)$ are elliptic in $z$, satisfying 	$
		\widehat{e}_{m}(-z)=(-1)^{m}\widehat{e}_{m}(z)$. 
		\item 
		The function $\widehat{e}_{m}(z),m\geq 1$ is smooth everywhere except for a simple pole along $z\in \Lambda_{\tau}$. 
		More precisely, the polar part of $\widehat{e}_{m}(z)$ at $z=0$ is
		\[
		{1\over z} {({\mathbold{Y}\bar{z}})^{m-1}\over (m-1)!}\,.
		\]
		\end{enumerate}
		\end{lem}
	\begin{proof}
	
		\begin{enumerate}[i).]
	
	\item 
	The ellipticity follows from that of their generating series $\widehat{S}_{c}(z)$ which can be proved directly using the automorphy of $\theta$.
	The parity follows from the parity of $\mathbold{A}(z), \theta(z)$ which gives 
	$\widehat{S}_{c}(-z)=-\widehat{S}_{-c}(z)$.
	\item 
	The claim that $\widehat{e}_{m}$ is smooth and has at worst simple poles along $z\in\Lambda_{\tau}$ on $E$ follows from the fact that the generating series $\widehat{S}_{c}(z)$ of the $\widehat{e}_{m}$'s is so.
The claim about its polar part follows from the following formula for the generating series of the polar parts of all $\widehat{e}_{m}$'s
\[
{1\over z}\cdot \left({\theta(z+c)\over \theta(c)}e^{c\mathbold{A}}\right)|_{z=0}
={1\over z}\cdot e^{c\mathbold{Y}\bar{z}}\,,
\]
where we have used the fact that $z,\bar{z}$ are independent complex quantities.

			\end{enumerate}
	\end{proof}

We next study some  differential and quadratic relations between the $\widehat{e}_{m}(z)$'s.

\begin{dfn}\label{dfnfm}
		Suppose $c_{i},i\in [N]$ are constants  which are not necessarily distinct. 
For a sequence
	\[
	\mathrm{m}=(m_{1},\cdots ,m_{i},\cdots,m_{N})\in \mathbb{N}^{N}\,,
	\]
define
\[
\widehat{f}_{\mathrm{m}}:=\prod_{i=1}^{N}\widehat{e}_{m_{i}}(z+c_{i})\,,\quad 	|\mathrm{m}|:=\sum_{i=1}^{N}m_{i}\,.
\]
Define also	
\[
\mathrm{m}\ominus 1:=
\bigoplus_{i=1}^{N} (m_{1},\cdots,m_{i}-1,\cdots, m_{N})\,,\quad 
\mathrm{m}\ominus \ell:=	\mathrm{m}\underbrace{\ominus 1\ominus 1\cdots \ominus 1}_{~\ell~\text{times}}\,~\text{for}~	\ell\in \mathbb{N}\,.
\]	

\end{dfn}

\begin{lem}	\label{lempartialrelation}
	Let the notation be as above.
	\begin{enumerate}[i).]
		\item 
		One has
		\[
		{\partial}_{\bar{z}}\widehat{e}_{m}(z)=\mathbold{Y}\,\widehat{e}_{m-1}(z)\,,\quad m\geq 1\,.
		\]

		\item
		
		Suppose $m_{i}\geq 1, i\in [N]$ and $c_{i},i\in [N]$ are constants which are not necessarily distinct. 
		Denote $\mathrm{r}=(m_{1},\cdots, m_{N-1})\in \mathbb{N}^{N-1}$.
		Then 
		\begin{eqnarray*}\label{eqnbarpartialprimitiveYfm}
			\mathbold{Y}\widehat{f}_{\mathrm{m}}&=&
		{\partial}_{\bar{z}}
		\left(	\sum_{\ell\geq 0}
			(-1)^{\ell}
			\widehat{f}_{\mathrm{r}\ominus\ell}\cdot 
			\widehat{e}_{m_{N}+\ell+1}\right)\\
			&=&	{\partial}_{\bar{z}}\left(
			\sum_{k_{1},\cdots, k_{N-1}}
			(-1)^{\sum_{i=1}^{N-1}k_{i}}
			\prod_{i=1}^{N-1}
			\widehat{e}_{m_{i}-k_{i}}(z+c_{i})\cdot
			\widehat{e}_{m_{N}+\sum_{i=1}^{N-1}k_{i}+1}(z+c_{N})\right)
			\,.
		\end{eqnarray*}
		\item

		Assume $x\neq y$ modulo $\Lambda_{\tau}$,
		then one has
		\begin{eqnarray*}
		\widehat{e}_{i}(x)\widehat{e}_{j}(y)
		=\sum_{k+\ell=i+j}{\ell-1\choose i-k}\widehat{e}_{k}(x-y)\widehat{e}_{\ell}(y)
		+\sum_{k+\ell=i+j}{k-1\choose i-1}\widehat{e}_{k}(x)\widehat{e}_{\ell}(y-x)\,,
		\end{eqnarray*}
	where ${b\choose a},a,b\in \mathbb{Z}$ are the generalized binomial numbers, e.g., 
	${-1\choose a}=(-1)^{a}$.
	
		In particular, one has for $m\geq 1$		
		\[
	\widehat{e}_{1}(x)\widehat{e}_{m}(y)=
	\widehat{e}_{1}(x-y)\widehat{e}_{m}(y)
	+m\widehat{e}_{m+1}(y)
	+\sum_{k+\ell=1+m}\widehat{e}_{k}(x)\widehat{e}_{\ell}(y-x)
	\,.
	\]
		
	\end{enumerate}
	
\end{lem}

\begin{proof}
	\begin{enumerate}[i).]
		\item 
		The first part of the claim is essentially  \cite[Equation 3.6]{Brown:2011}. It follows by comparing the
		Laurent coefficients
		of the following identity, proved
		by a direct computation using the definition of 	$\widehat{S}_{c}(z)$ given in Definition \ref{dfnSzegokernel} and the fact that $S_{c}(z)$ is meromorphic in $z$,
		\begin{equation}\label{eqnbarpartialhatScz} 
			{\partial}_{\bar{z}}	(	\widehat{S}_{c}(z))=c
			{\partial}_{\bar{z}}\mathbold{A}     \cdot \widehat{S}_{c}(z)
			=c
			\mathbold{Y} \widehat{S}_{c}(z)\,.
		\end{equation}

	\item 
		From (i) and the product rule, we have
	\[
	{\partial}_{\bar{z}}\widehat{f}_{\mathrm{r}}=
	\widehat{f}_{\mathrm{r}\ominus 1}\,.
	\]
	Integration by parts gives
	\[
	\mathbold{Y}\widehat{f}_{\mathrm{r}}\widehat{e}_{m_{N}}
	=\widehat{f}_{\mathrm{r}}\cdot {\partial}_{\bar{z}}\widehat{e}_{m_{N}+1}
	={\partial}_{\bar{z}}(\widehat{f}_{\mathrm{r}}\widehat{e}_{m_{N}+1})-
	\mathbold{Y}\widehat{f}_{\mathrm{r}\ominus 1}\widehat{e}_{m_{N}+1}\,.
	\]
	Applying this identity iteratively to the part involving $\mathbold{Y}$ and noting that $\widehat{f}_{\mathrm{r}\ominus \ell}=0$ for $\ell>|\mathrm{r}|$ (since $\widehat{e}_{m}=0$ for $m<0$), we obtain the 
	first identity. The second identity follows directly.

		\item 
		This is part of the quadratic Arnold-type relations given in \cite[Equation 3.7]{Brown:2011} which are derived by using Fay's trisecant identity.
		To be more precise, Fay's trisecant identity gives
		\[
		S_{a}(x)S_{b}(y)=S_{a}(x-y)S_{a+b}(y)+S_{a+b}(x)S_{b}(y-x)\,,
		\]
		and thus the identity for the $\widehat{S}$ version.
		Multiplying the resulting identity throughout by $ab(a+b)$ and 
		comparing the Taylor coefficients in $a,b$, one obtains
		\[
		\widehat{e}_{i}(x)\widehat{e}_{j}(y)=-\widehat{e}_{i-1}(x)\widehat{e}_{j+1}(y)
		+\sum_{k+\ell=i+j}
		{\ell\choose i-k} \widehat{e}_{k}(x-y)\widehat{e}_{\ell}(y)
		+\sum_{k+\ell=i+j}
		{k\choose i-1} \widehat{e}_{k}(x)\widehat{e}_{\ell}(y-x)\,.
		\]

		Iterating, one has
		\begin{eqnarray*}
			\widehat{e}_{i}(x)\widehat{e}_{j}(y)&=&
			\sum_{r=0}^{i}(-1)^{r}
			\sum_{k+\ell=i+j}
			\left(
			{\ell\choose i-r-k} \widehat{e}_{k}(x-y)\widehat{e}_{\ell}(y)+{k\choose i-r-1} \widehat{e}_{k}(x)\widehat{e}_{\ell}(y-x)\right)\\
			&=&
			\sum_{k+\ell=i+j}\widehat{e}_{k}(x-y)\widehat{e}_{\ell}(y)
			\sum_{r=0}^{i}(-1)^{r}
			{\ell\choose i-r-k} \\
			&+&
			\sum_{k+\ell=i+j} \widehat{e}_{k}(x)\widehat{e}_{\ell}(y-x)\sum_{r=0}^{i}(-1)^{r}{k\choose i-r-1}\,.
		\end{eqnarray*}
	Simplifying using the elementary identity $\sum_{r=0}^{i}(-1)^{r}{N\choose r}=(-1)^{i}{N-1\choose i}$ for $N\geq 0, i\geq 0$, we obtain the desired claim.
	Alternatively, one can prove the claim by comparing the Laurent coefficients 
		in Fay's trisecant identity.
		
	\end{enumerate}
\end{proof}

\begin{rem}\label{remrdifferentialandquadraticrelations}~
	\begin{enumerate}[i).]
	\item 
	The generating series technique also gives rise to formulas expressing $\partial_{z}\widehat{e}_{m}$
	in terms of elements in $\widehat{\mathfrak{F}}$, generated by $\{\widehat{e}_{m}\}_{m\geq 1}$ (or equivalently $\{\widehat{\mathcal{E}}^{*}_{m}\}_{m\geq 1}$),
	that are parallel to \eqref{eqnEm}.
	To be more explicit, 
	from \eqref{eqngeneratingseriesofEms} and \eqref{eqnEm*}, \eqref{eqnEm} we obtain 
	\begin{eqnarray*}
		\partial_{z}\widehat{S}_{c}(z)
		&=&\widehat{S}_{c}(z)\cdot \left(
		-c\mathbold{Y}+\sum_{m\geq 1}(-1)^{m} \mathcal{E}_{m+1} c^{m}
		\right)\\
		&=&{1\over c}
		\sum_{m\geq 0}{c^{m}}\widehat{e}_{m}
		\cdot \left(
		\sum_{m\geq 1}{c^{m}\over m!} (\widehat{\mathcal{E}}_{m+1}^{*}-m !\cdot 2\widehat{G}_{m+1}-\delta_{m,0}\mathbold{A}(z))
		\right)\,.
	\end{eqnarray*}
	Comparing Laurent coefficients in $c$ gives a desired relation
	\begin{equation}\label{eqnpartialzem}
		\partial_{z}\widehat{e}_{m}(z)
		=\sum_{\substack{a+b=m\\ b\geq 1}}\widehat{e}_{a}(z)\cdot  {1\over b!}(\widehat{\mathcal{E}}_{b+1}^{*}-b!\cdot 2\widehat{G}_{b+1})\,.
	\end{equation}
	This lands in the ring 
	$
	\widehat{M}\otimes\widehat{\mathfrak{F}}
	$,
	where $\widehat{M}$ is the ring of almost-holomorphic modular forms \cite{Kaneko:1995}
	with the generators $\{\widehat{G}_{k}\}_{k\geq 1}$ introduced in Definition \ref{dfnhatthetazeta}.	
	The ring $\widehat{M}$ is
	 isomorphic to $\mathbb{C}[\widehat{G}_{2}, \widehat{G}_{4},\widehat{G}_{6}]$ after quotienting out the relations among the generators.

		\item 
	Setting $m=1$ in Lemma \ref{lempartialrelation} (iii) and
	using Definition \ref{dfnEisensteinKroneckerseries} and  \eqref{eqnfirstfewhatems}, 
	one arrives at the familiar addition formula
	\[
	(\zeta(x)+\zeta(y)+\zeta(-x-y))^{2}
	=	\wp(x)+\wp(y)+\wp(-x-y)\,.
	\]
		\end{enumerate}
\end{rem}

\section{Iterated integrals of Eisenstein-Kronecker forms}
\label{seciteratedintegrals}

	In this part we compute iterated regularized integrals of Eisenstein-Kronecker forms and discuss their connections to Chen's iterated path integrals.\\

We first  construct some differential forms in terms of the  linear coordinate $z$ on the universal cover $\mathbb{C}$ of $E$.
A holomorphic volume form on $E$ is given by $dz$, and a volume form $\mathrm{vol}$ satisfying $\int_{E}\mathrm{vol}=1$ is
\begin{equation}
	\mathrm{vol}:= { \mathbf{i} \over 2 \im \tau}dz\wedge d\bar z\,.
\end{equation}
We also denote 
\begin{equation}\label{eqnpsibeta}
	\beta={d\overline{z}-dz\over \bar{\tau}-\tau}\,.
\end{equation}
Note that the cohomology class of $\beta$ in $H^{1}(E,\mathbb{C})$ is the Poincar\'e dual of the $A$-cycle class in the weak Torelli marking $\{[A],[B]\}$.\\

Recall that  $A_{E}^{1}(\star D)$ is the space of 1-forms  
that are smooth everywhere on $E$ except for possible holomorphic poles along an effective divisor $D$ on $E$.
One can introduce a notion \cite[Definition 2.3]{Felder:2017} (see also \cite[Definition 2.11]{Li:2020regularized}) of holomorphic residue $\mathrm{res}_{\partial}$
that acts on $A_{E}^{1}(\star D)$.
Concretely, this is given by
\begin{equation}\label{eqnholomorphicresidue}
	\mathrm{res}_{\partial}={1\over 2\pi \mathbf{i}}\lim_{\varepsilon\rightarrow 0}\int_{\partial B_{\varepsilon}(D)}\quad \,,
\end{equation}
where $B_{\varepsilon}(D)$ is a disk bundle of radius $\varepsilon$ over the divisor $D$. The result is independent  of the choice of
the local defining equation for $D$ and the Hermitian metric that is compatible with the complex structure.
In particular, from the description of the right hand side of \eqref{eqnholomorphicresidue} one has 
\cite[Lemma 2.2]{Li:2020regularized}
\begin{equation}\label{eqnvanishingofholomorphicresiduebytypereasonscaseDol0}
	\mathrm{res}_{\partial}(\omega)=0\,,\quad \forall\, \omega\in A^{0,1}_{E}(\star D)\,.
\end{equation}
An analytic notion of regularized integral 
\begin{equation}
\dashint_{E}: \, A_{E}^{2}(\star D)\rightarrow \mathbb{C}
\end{equation}
is introduced \cite{Li:2020regularized}, extended and further developed in \cite{Li:2022regularized}.
It extends the ordinary notion of integral on smooth forms and provides a nice regularization scheme that enjoys many nice properties such as the Fubini theorem.
A purely cohomological formulation of this notion is provided in \cite[Theorem A]{Zhou:2023cohomologicalpairings}.\\


For a function $f\in \Omega^{0}_{E}(\star D)$,
by slightly deforming the $A$-cycle $A$ if necessary, we can always assume that $A$ has no intersection with $D$.
When no confusion might arise we adapt short-hand notation such as
\begin{equation}
\dashint_{E} f:=\dashint_{E} f\,\mathrm{vol}\,,\quad  \int_{A}f:=\int_{A}f dz\,.
\end{equation}
Similarly, we write
$\mathrm{res}_{\partial,\,z=p}(f)$
for the residue  $\mathrm{res}_{\partial,\,z=p}(f dz)$  of the corresponding differential $fdz$ at $z=p$, and write $\mathrm{res}_{\partial}(f)$ for the summation of all local residues.

Set $[n]:=(1,2,\cdots, n)$, denote $E_{[n]}:=E_{n}\times E_{n-1}\times\cdots\times E_{1}$, where 
$E_{k}$ represents a copy of elliptic curve equipped with the corresponding linear coordinate $z_{k}, k\in [n]$.
For a function $f\in \Omega_{E_{[n]}}^{0}(\star D)$, we use the following short-hand notation for the iterated residues, iterated $A$-cycle integrals, and iterated regularized integrals
\begin{eqnarray}\label{eqnconventionforiteratedintegrals}
	\mathrm{res}_{\partial}(f)&:=&\mathrm{res}_{\partial}\,dz_{n}\cdots
	\mathrm{res}_{\partial}\,dz_{2} \,\mathrm{res}_{\partial}\,(fdz_{1}) \,,\\\quad 
	\int_{A_{\sigma([n])}}f&:=&\int_{A_{\sigma(n)}}\,dz_{\sigma(n)}\cdots
	\int_{A_{\sigma(2)}}\,dz_{\sigma(2)} \int_{A_{\sigma(1)}}f\,dz_{\sigma(1)} \,,\nonumber\\
	\dashint_{E_{\sigma([n])}}f
	&:=&	\dashint_{E_{\sigma(n)}}\mathrm{vol}_{\sigma(n)}\cdots \dashint_{E_{\sigma(2)}}\mathrm{vol}_{\sigma(2)} \dashint_{E_{\sigma(1)}}f\,\mathrm{vol}_{\sigma(1)}\,,\quad \sigma\in \mathfrak{S}([n])\nonumber\,.
\end{eqnarray}
Here the integration domain $A_{\sigma([n])}$ is constructed such that it avoids the singularities of $f$, with
\[
A_{i}=\{\varepsilon_{i}\tau+\tau+t~|~t\in [0,1]\}\,,\quad \varepsilon_{\sigma(n)}<\cdots<\varepsilon_{\sigma(2)}<\varepsilon_{\sigma(1)}\,,
\]
for sufficiently small positive real numbers $\varepsilon_{i},i\in [n]$.
Due to the meromorphicity of $f$ and the residue theorem, the iterated $A$-cycle integrals are invariant
under small deformations of the  $\varepsilon_{i}$ parameters. The integration domain in $\dashint_{E_{\sigma([n])}}$ can also be constructed explicitly \cite[Section 3.4]{Li:2020regularized}. and the regularized integral extends to an operator on 
$f\in A_{E_{[n]}}^{0}(\star D)$  \cite[Section 2.5]{Li:2020regularized}.	
Note that unlike the iterated residues and iterated $A$-cycle integrals, the iterated regularized integral in \eqref{eqnconventionforiteratedintegrals} is independent of the ordering of the iteration \cite[Theorem 2.37]{Li:2020regularized} and is therefore simply called the regularized integral.

\subsection{Eisenstein-Kronecker forms}
\label{secEKforms}

Since by Lemma \ref{lempolarparts} one has
$
\widehat{e}_{m}(z)\in A_{E}^{0}(\star 0)$,
one can then consider their holomorphic residues and regularized integrals, as defined  in \cite{Li:2020regularized} and further developed in
\cite{Li:2022regularized, Zhou:2023cohomologicalpairings}.\\

We first establish the following simple result which 
indicates
that the set of generators $\{\widehat{e}_{m}(z)\}_{m\geq 1}$, which have at worst simple poles,
is an ideal one (in contrast to e.g., $\{\widehat{\mathcal{E}}^{*}_{m}(z)\}_{m\geq 1}$) for exact computations.

	\begin{dfn}\label{dfnEKforms}
		Let the notation and conventions be as above.
		Consider the following subvector space of the algebra $\widehat{\mathfrak{F}}$ generated by $\{\widehat{e}_{m}(z)\}_{m\geq 1}$, consisting of the linear span of degree $0$ and $1$ elements, 
\begin{equation}
	\mathcal{V}=\sum_{m\geq 0}\mathbb{C}\,\widehat{e}_{m}(z)\,.
\end{equation}
		We call $	\mathcal{V}dz, 	\mathcal{V}\mathrm{vol}$ the space of Eisenstein-Kronecker forms.
	\end{dfn}

\begin{lem}\label{lemresidueregularizedintetgralofhatem}
Let the notation be as above.
Then one has
\[
\mathrm{res}_{\partial}\,(\widehat{e}_{m}(z)dz)=\delta_{m,1}\,,\quad
\dashint_{E}\,\widehat{e}_{m}(z)\mathrm{vol}=\delta_{m,0}\,.
\]
\end{lem}
\begin{proof}
We prove the statements by computing the holomorphic residue $\mathrm{res}_{\partial}$ and regularized integral $\dashint_{E}$ of the generating series 
$\widehat{S}_{c}(z)$ of $\widehat{e}_{m}(z),m\geq 0$.
The statement on the holomorphic residue follows from Lemma \ref{lempolarparts}.
For the regularized integral, from  \eqref{eqnbarpartialhatScz} we have
\[
\bar{\partial}(\widehat{S}_{c}(z)dz)=c\cdot \pii \beta\wedge \widehat{S}_{c}(z)dz
=-c\cdot \pii  \widehat{S}_{c}(z)dz\wedge \psi
=-c\cdot \pii  \widehat{S}_{c}(z)\mathrm{vol}\,,
\]
and thus
\[
\dashint_{E}\,\widehat{S}_{c}(z)\mathrm{vol}=-{1\over \pii c}\dashint_{E}\bar{\partial}(\widehat{S}_{c}(z)dz)\,.
\]
According to the Cauchy integral formula \cite[Theorem 2.14]{Li:2020regularized} for the regularized integral (see also \cite[Section 2.2.2]{Zhou:2023cohomologicalpairings} for a review),
this gives 
\[
\dashint_{E}\,\widehat{S}_{c}(z)\mathrm{vol}={1\over c}\mathrm{res}_{\partial}\,(\widehat{S}_{c}(z)dz)={1\over c}
\mathrm{res}_{\partial} \, ({1\over z}e^{c\mathbold{A}}dz) \,.
\]
The desired result then follows from the following fact proved 
using \eqref{eqnholomorphicresidue}
\begin{equation}\label{eqnvanishingofzbarz-ndz}
\mathrm{res}_{\partial} \, ({\bar{z}^{k}\over z^{n}}dz)=0\,,\quad k\geq 1\,.
\end{equation}

\end{proof}

\begin{rem}\label{remholomorphicresidueofhatBm}
	Note that a smooth function with holomorphic poles can be even under the action $(z,\bar{z})\mapsto (-z,-\bar{z})$
	and have a  simple pole at $z=0$.
	For example, from \eqref{eqnfirstfewhatems} and \eqref{eqnthetazetaexpansion}, one can see that locally near $z=0$  
	one has 
	\[
	\widehat{e}_{1}^{2}(z)=\widehat{Z}^{2}(z)=
	{1\over z^2}+2\mathbold{Y} {\bar{z}\over z}+\cdots\,,
	\quad 
		\widehat{e}_{2}(z)={1\over 2}\cdot (2\mathbold{Y} {\bar{z}\over z}+\cdots)\,.
	\]
	Both of them have trivial holomorphic residues under $\mathrm{res}_{\partial}$.
	
	In fact, according to \eqref{eqnvanishingofzbarz-ndz},
	as long as the holomorphic residue operation $\mathrm{res}_{\partial}$ is concerned, monomials in the $\widehat{e}_{m}(z)$'s behave
	as if they were ordinary meromorphic functions. In particular,  from Lemma \ref{lempolarparts} and \eqref{eqnthetazetaexpansion} we have
	\begin{equation}\label{eqnresidueofemr}
\mathrm{res}_{\partial}\,(\widehat{e}_{m}^{N}dz)=0\,,\quad m\geq 2\,,\quad \quad 
\mathrm{res}_{\partial}\,(\widehat{e}_{1}^{N}dz)=
\sum_{\substack{(k_{1},\cdots, k_{N})\\ \sum_{i=1}^{N} (2k_{i}-1)=-1}}
\prod_{i=1}^{N}(-2\widehat{G}_{2k_{i}})\,,
\end{equation}
where we have set the convention $-2G_{0}=-1$.
\end{rem}

We  shall also study iterated residues and iterated regularized integrals of almost-elliptic functions of the following type.

\begin{dfn}\label{dfnfunctionspaceassociatedtoD}
	Let the notation be as above.
 Fix a set of generic values $c_{k,\ell},k,\ell\in [n]$.
	Consider the smooth hypersurface arrangement $\mathsf{D}$ in $E_{[n]}$ given by (the image under the quotient $\mathbb{C}_{[n]}\rightarrow E_{[n]}$)
\begin{equation}
	\mathsf{D}=\bigcup_{\substack{k,\ell \in [n]\\
		k\neq \ell}} \big\{s_{k,\ell}=0\big\}\,,\quad s_{k,\ell}:=c_{k,\ell}+z_{k}-z_{\ell}\,.
\end{equation}
	We also regard $s_{k,l}$ as defining a morphism $E_{[n]}\rightarrow E$.
	Define the multiplication map (here $U=E_{[n]}-\mathsf{D}$)
	\[
	\mu:
	\bigotimes_{k,\ell\in [n]}s_{k,\ell}^{*}\widehat{ \mathcal{R}}
	\longrightarrow C^\infty(U),\qquad
	\bigotimes_{k,\ell}f_{k,\ell}
	\longmapsto
	\prod_{k,\ell}f_{k,\ell}\,,
	\]
	and consequently the following spaces	
\begin{equation}
\mathcal{F}_{\mathsf{D}}=\mu\left(\bigotimes_{k,\ell\in [n]} s_{k,\ell}^{*}\widehat{\mathfrak{F}}\right)\,,\quad 
\mathcal{V}_{\mathsf{D}}=\mu\left(\bigotimes_{k\in [n]}\left(\bigoplus_{\ell\in [n]} s_{k,\ell}^{*}\mathcal{V}\right)\right)\,,\quad
\mathcal{V}_{\mathsf{D}}^{\circ}=\mu\left(\bigotimes_{k\in [n]}s_{k,k+1}^{*}\mathcal{V}\right)\,,
\end{equation}
with the conventions $k\neq \ell, z_{n+1}=z_{1},c_{n,n+1}=c_{n,1}$.
In this work we call the latter two spaces of Eisenstein-Kronecker forms associated to $\mathsf{D}\subseteq E_{[n]}$, with convention \eqref{eqnconventionforiteratedintegrals} understood.

	\end{dfn}

One has the generating sets of the vector spaces $\mathcal{F}_{\mathsf{D}}, \mathcal{V}_{\mathsf{D}}, \mathcal{V}_{\mathsf{D}}^{\circ}$ given in Table	\ref{table-bases}.
\begin{table}[h]
	\centering
	\caption{Generating sets for the vector spaces $
	\mathcal{F}_{\mathsf{D}}, 
		\mathcal{V}_{\mathsf{D}}, \mathcal{V}_{\mathsf{D}}^{\circ}$ \,.}
	\label{table-bases}
	\renewcommand{\arraystretch}{1.5} 
	\begin{tabular}{c|c}
		\hline
		vector space &  generators: $k,\ell\in [n]\,, \quad k\neq \ell$ \\
		\hline
		$\mathcal{F}_{\mathsf{D}}$  &  $\prod_{k=1}^{n}\prod_{\ell=1}^{n}\widehat{e}_{m_{k,\ell}}^{r_{k,\ell}}(s_{k,\ell})$,\quad $r_{k,\ell}\geq 0$	\\	
		$\mathcal{V}_{\mathsf{D}}$		&	$\prod_{k=1}^{n}\prod_{\ell=1}^{n}\widehat{e}_{m_{k,\ell}}^{r_{k,\ell}}(s_{k,\ell})$,\quad  $r_{k,\ell}\in \{0,1\}\,,\quad 
		\sum_{\ell=1}^{n}r_{k,\ell}\leq 1$\,	\\	
		$\mathcal{V}_{\mathsf{D}}^{\circ}$	& $\prod_{k=1}^{n}\widehat{e}_{m_{k,k+1}}^{r_{k,k+1}}(s_{k,k+1})$,\quad $r_{k,\ell}\in \{0,1\}$\,	\\	
		\hline
	\end{tabular}
\end{table}

Many constructions and results regarding these spaces are most easily formulated and visualized by the so-called \emph{indicating graphs} which we now introduce.

\begin{dfn}\label{dfnindicatinggraph}
	For any monomial $f=\prod_{k,\ell=1}^{n}\widehat{e}_{m_{k,\ell}}^{r_{k,\ell}}(s_{k,\ell})\in \mathcal{F}_{\mathsf{D}}$, its indicating graph $\Gamma(f)$ is defined to be the following  graph:
	\begin{itemize}
		\item it has $n$ vertices labeled by $k\in [n]$,
		\item for each factor $\widehat{e}_{m_{k,\ell}}^{r_{k,\ell}}(s_{k,\ell})$ with $k\neq \ell$, assign $r_{k,\ell}$ oriented edges starting at $k$ and ending at $\ell$ if $m_{k,\ell}\geq 1$,  and no edges if $m_{k,\ell}=0$. See Figure \ref{figure:indicatinggraphs} below for an illustration.
	\end{itemize}
	 \begin{figure}[h]
 	\centering
 		\centering
 	\includegraphics[scale=1]{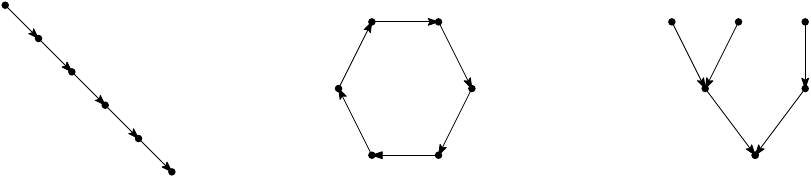}
 	\caption{Indicating graphs that are chain, loop, tree, respectively. Here the labellings of the vertices are omitted.}
 	\label{figure:indicatinggraphs}
  \end{figure}
  
For  each vertex $k$, let $v_{k,\mathrm{in}}, v_{k,\mathrm{out}}$
be the inner and outer valencies at $k$.
Since each edge contributes one inner valency and one outer valency, we have
\begin{equation}\label{eqntotalinnervalencies=totaloutervalencies}
\sum_{k=1}^{n} v_{k,\mathrm{in}}= \sum_{k=1}^{n} v_{k,\mathrm{out}}\,.
\end{equation}
	\end{dfn}

The geometric data of the monomial $f$ is conveniently encoded in the indicating graph $\Gamma(f)$. For example:
\begin{itemize}
	\item 
 the number of edges, given by $\delta_{m_{k,\ell}\geq 1}\cdot r_{k,\ell}$,
 corresponds to	the degree of the monomial $\widehat{e}_{m_{k,\ell}}^{r_{k,\ell}}(s_{k,\ell})$. 
This is also the \emph{multiplicity} of the irreducible component $(s_{k,\ell}=0)$ of the polar divisor of $\widehat{e}_{m_{k,\ell}}^{r_{k,\ell}}(s_{k,\ell})$. In particular,  $\mu(\bigotimes_{k,\ell\in [n]} s_{k,\ell}^{*}\mathcal{V})$  is the subspace
of $\mathcal{F}_{\mathsf{D}}$
that consists of elements whose corresponding polar divisors are reduced.
This space is of great interest in that it underlies 
a model for the cohomology of $E_{[n]}-\mathsf{D}$,
as shown in \cite{Brown:2011}.  
	\item
	the outer valency at the vertex $k$ is given by
	$ v_{k,\mathrm{out}}=\sum_{ \ell\neq k} \delta_{m_{k,\ell}\geq 1} r_{k,\ell}$.
	In particular, 
	\begin{equation}\label{eqnoutervalenciesforelementsinVD}
		f\in \mathcal{V}_{\mathsf{D}}\quad \Rightarrow \quad  v_{k,\mathrm{out}}\leq 1\,~\text{for}~ k\in [n]\,.
		\end{equation}

	\item 
	decomposition of $\Gamma(f)$ into connected components corresponds to 
	factorization of $f$ and thus also those of the operations $\mathrm{res}_{\partial},\dashint_{E_{[n]}}$ on $f$.
	In particular, the existence of a trivial connected component (i.e., one with only one vertex $b$ and no edges) represents the fact that $f$ is independent of $z_{b}$ and thus the $\dashint_{E_{b}}$ part in $\dashint_{E_{[n]}}$
gives the identity operator.\footnote{Thanks to the fact that iterated regularized integral
 $\dashint_{E_{[n]}}f$ is independent of the ordering for the iteration \cite[Theorem 2.37]{Li:2020regularized}.} 
	\end{itemize}

\subsection{Iterated residues and iterated regularized integrals}
\label{seccomputationsoniteratedresiduesandintegrals}

We next study iterated residues and regularized integrals for elements in $\mathcal{V}_{\mathsf{D}}$.
The more general case $\mathcal{F}_{\mathsf{D}}$ will not be discussed here since 
the combinatorics is more involved 
	(and can be described using rooted labeled forests as exhibited in \cite{Li:2022regularized}).

\subsubsection{Reduction from $\mathcal{V}_{\mathsf{D}}$ to $\mathcal{V}_{\mathsf{D}}^{\circ}$}

We first show the following simple fact based on combinatorial properties of the indicating graphs.

\begin{lem}\label{lemreduceVDtoVD0}
Let the notation be as above.
	\begin{enumerate}[i.)]
		\item 
		Let the notation be as above.
			Let $f=\prod_{k=1}^{n}\prod_{\ell=1}^{n}\widehat{e}_{m_{k\ell}}^{r_{k\ell}}(s_{k\ell})\in \mathcal{F}_{\mathsf{D}}$.
	  If there exists some vertex $b$ in $\Gamma(f)$ such that
	\[
	v_{b,\mathrm{in}}+v_{b,\mathrm{out}}=1\,,
	\]
	then $\dashint_{E_{[n]}}f=0$.
	\item 	
	Let $f=\prod_{k=1}^{n}\prod_{\ell=1}^{n}\widehat{e}_{m_{k\ell}}^{r_{k\ell}}(s_{k\ell})\in \mathcal{V}_{\mathsf{D}}$.
	Then 
	\[
	\dashint_{E_{[n]}}f\neq 0
	\]
	 only if any
	connected component of the indicating graph $\Gamma(f)$
	is  trivial (i.e., has no edges) or a loop.
	\end{enumerate}
	\end{lem}
\begin{proof}
		\begin{enumerate}[i.)]
			
			\item 		
					
			Denote the other vertex of the unique edge at $b$ by $c$, then 
			the function $f$ must have the form
			\[
			f=\widehat{e}_{m}(s_{bc})g\quad \text{or} \quad \widehat{e}_{m}(s_{cb})g\,,
			\]
			for some $m\geq 1$ and function $g$ that is independent of $z_{b}$.
			Applying Lemma \ref{lemresidueregularizedintetgralofhatem} we see that
			\[
			\dashint_{E_{b}}f \mathrm{vol}_{b}= 0\,.
			\]
			This gives the desired claim.
		\item

	By the factorization property of the iterated regularized integration, 
	one can assume that $\Gamma(f)$ is already connected.
	There is nothing to prove if it has no edges. Otherwise, from \eqref{eqntotalinnervalencies=totaloutervalencies} and \eqref{eqnoutervalenciesforelementsinVD} we have
	 \[
	\sum_{k=1}^{n}( v_{k,\mathrm{in}}+v_{k,\mathrm{out}})=2 \sum_{k=1}^{n} v_{k,\mathrm{out}}\leq 2n\,.
	\]
	
	If there exists a vertex, say $a$, with 
	\[
	v_{a,\mathrm{in}}+v_{a,\mathrm{out}}>2\,,
	\]
	then there must exist another vertex $b$ such that
	\[
	v_{b,\mathrm{in}}+v_{b,\mathrm{out}}\leq 1\,.
	\]
	Since $\Gamma(f)$ is connected, one can not have
	$ v_{b,\mathrm{in}}+v_{b,\mathrm{out}}=0$. Thus
	$
	v_{b,\mathrm{in}}+v_{b,\mathrm{out}}=1
	$ 
	and by (i) the regularized integral $\dashint_{E_{[n]}}f$ is zero.
	Therefore, $\dashint_{E_{[n]}}f$ is nonzero only if
	\[
	v_{a,\mathrm{in}}+v_{a,\mathrm{out}}=2\,,\quad a\in [n]\,.
	\]
	This also tells that the number of edges is $2n/2=n$.
	It is now easy to see (for example by computing the Euler characteristic in two ways)
	that $\Gamma(f)$ is a loop.
	
		\end{enumerate}
\end{proof}

Lemma \ref{lemreduceVDtoVD0} allows 
to reduce
the study of regularized integral for elements in $\mathcal{V}_{\mathsf{D}}$
to that for elements in $\mathcal{V}_{\mathsf{D}}^{\circ}$.
However, this simplification does not apply directly to iterated residues since the latter depend on the ordering of iteration.

\subsubsection{Evaluation of iterated residues and iterated regularized integrals}
\label{subsecevaluation}

For later simplification, we introduce the following short-hand notation.
\begin{dfn}
Let the notation be as above. Define 
\[
\widehat{e}_{\mathrm{m}}=\prod_{k=1}^{n}\widehat{e}_{m_{k}}(s_{k,k+1})\,,\quad 
\mathrm{m}=(m_{1},m_{2},\cdots, m_{n})\in \mathbb{N}^{n}\,.
\]
Then according to Definition \ref{dfnfunctionspaceassociatedtoD} one has
\begin{equation}
\mathcal{V}_{\mathsf{D}}^{\circ}=\bigoplus_{\mathrm{m}\in \mathbb{N}^{n}}\mathbb{C}\,\widehat{e}_{\mathrm{m}}\,.
\end{equation}
The
number of nonzero entries in the sequence $\mathrm{m}$
is called the length and denoted by $\ell(\mathrm{m})$, while the quantity $|\mathrm{m}|:=\sum_{k=1}^{m}m_{k}$ is called the size of  $\mathrm{m}$.
They 
correspond to the degree and grading of $\widehat{e}_{\mathrm{m}}\in \mathcal{F}_{\mathsf{D}}$, respectively.

\end{dfn}

 Following the methods developed in \cite{Li:2020regularized, Li:2022regularized}, we prove the following results.
 	
 \begin{prop}\label{propiteratedcalculation}
 	Let the notation and conventions be as above. Let $n\geq 2$ and $z_{0}$ be a constant.
	Define
\[
	\xi_{\mathrm{chain}}=
\prod_{k=1}^{n-1}\widehat{e}_{m_{k}}(s_{k,k+1})\cdot \widehat{e}_{m_{n}}(z_{n}-z_{0})\,,\quad
\xi_{\mathrm{loop}}= \widehat{e}_{\mathrm{m}}\,.
\]
 	\begin{enumerate}[i).]
 		\item 
 		One has
 	\[
	\mathrm{res}_{\partial}\,(	\xi_{\mathrm{chain}}\cdot\bigwedge_{k=1}^{n}dz_{k})=\prod_{k=1}^{n}\delta_{m_{k},1}\,,\quad
 		\mathrm{res}_{\partial}\,(\xi_{\mathrm{loop}}\cdot\wedge_{k=1}^{n}dz_{k})=0\,.
\]
 		\item
 		One has
 		\begin{eqnarray*}	
 		\dashint_{E_{[n]}}	\xi_{\mathrm{chain}}\cdot \bigwedge_{k=1}^{n}\mathrm{vol}_{k}&=&\prod_{k=1}^{n}\delta_{m_{k},0} \,,\\  	 		
 			\dashint_{E_{[n]}}\xi_{\mathrm{loop}}\cdot\bigwedge_{k=1}^{n}\mathrm{vol}_{k}&=&
 			 			\widehat{e}_{|\mathrm{m}|}(\sum_{k=1}^{n}s_{k,k+1})\cdot
 			 			\left(	\prod_{k=1}^{n}\delta_{m_{k},0}-
 			 			\prod_{k=1}^{n}(\delta_{m_{k},0}-1)\right)\,.
 		 		\end{eqnarray*}

 	\end{enumerate}
 \end{prop}
The nomenclature of $\xi_{\mathrm{chain}},\xi_{\mathrm{loop}}$ 
originates from the types of the indicating graphs as shown in  Figure \ref{figure:indicatinggraphs}.

 \begin{proof}
  	\begin{enumerate}[i).]
 	\item 
 	The first relation follows by direct computations using  Lemma \ref{lemresidueregularizedintetgralofhatem}.

For the second relation, consider $C_{\mathrm{m}}=\mathrm{res}_{\partial}\,(\widehat{e}_{\mathrm{m}}\cdot\wedge_{k=1}^{n}dz_{k})$.
 	From
Lemma \ref{lemresidueregularizedintetgralofhatem} and the expression for $s_{k,l}$ given in Definition \ref{dfnfunctionspaceassociatedtoD}, we see that
\begin{eqnarray}\label{eqnresidueofproductoftwo}
	\mathrm{res}_{\partial}\,( \widehat{e}_{m_{1}}(s_{12})\widehat{e}_{m_{n}}(s_{n1})dz_{1})&=&\delta_{m_{1},1}\cdot\widehat{e}_{m_{n}}(s_{n1})|_{s_{12}=0}
	+(-1)^{m_{n}}\delta_{m_{n},1}\cdot\widehat{e}_{m_{1}}(s_{12})|_{s_{n1}=0}\nonumber\\
	&=& \delta_{m_{1},1}\cdot\widehat{e}_{m_{n}}(s_{n1}+s_{12})
	-\delta_{m_{n},1}\cdot\widehat{e}_{m_{1}}(s_{n1}+s_{12})\,.
\end{eqnarray}
Then using \eqref{eqnresidueofproductoftwo} we have the recursion
\[
C_{(m_1,m_{2},\cdots, m_{n})}=\delta_{m_{1},1}C_{(m_{2},\cdots, m_{n-1}, m_{n})}-\delta_{m_{n},1}C_{(m_{2},\cdots, m_{n-1},m_{1})}\,.
\]
The values for the initial case $n=2$ are
\[
C_{(m_1,m_{2})}=\mathrm{res}_{\partial}\,dz_{2}\left(\delta_{m_{1},1}\cdot\widehat{e}_{m_{n}}(s_{n1}+s_{12})
-\delta_{m_{n},1}\cdot\widehat{e}_{m_{1}}(s_{21}+s_{12})\right)=0\,,
\]
since $s_{21}+s_{12}$ is a constant independent of $z_{2}$.
Iterating, one proves  the desired claim.

 	\item
 	 	The first relation follows from Lemma \ref{lemreduceVDtoVD0}.

 		For the second relation, using 
 		Lemma \ref{lempolarparts} (i), Lemma \ref{lempartialrelation} (ii),  and the Cauchy integral formula \cite[Theorem 2.14]{Li:2020regularized} for the regularized integral as in Lemma \ref{lemresidueregularizedintetgralofhatem}, 
 		we have 
 	\begin{eqnarray*}
 		\dashint_{E_{1}}\widehat{e}_{m_{1}}(s_{12})\widehat{e}_{m_{n}}(s_{n1})\mathrm{vol}_{1}&=& (-1)^{m_{n}}
 			\dashint_{E_{1}}\widehat{e}_{m_{1}}(s_{12})\widehat{e}_{m_{n}}(-s_{n1})\mathrm{vol}_{1}\\
 			&=& (-1)^{m_{n}}\cdot {-1\over \pii}
 			\dashint_{E_{1}}
 			\bar{\partial}
 		\left(\sum_{k=0}^{m_{1}}(-1)^{k}\widehat{e}_{m_{1}-k}(s_{12})\widehat{e}_{m_{n}+1+k}(-s_{n1})dz\right)\\
 			&=&	 (-1)^{m_{n}}\mathrm{res}_{\partial}\,\left(\sum_{k=0}^{m_{1}}(-1)^{k}\widehat{e}_{m_{1}-k}(s_{12})\widehat{e}_{m_{n}+1+k}(-s_{n1})dz\right)\\ 			&=&	 -\mathrm{res}_{\partial}\,\left(\sum_{k=0}^{m_{1}}\widehat{e}_{m_{1}-k}(s_{12})\widehat{e}_{m_{n}+1+k}(s_{n1})dz_{1}\right) 			\,.
 \end{eqnarray*}

Using \eqref{eqnresidueofproductoftwo}, this gives
\begin{equation}\label{eqnregularizedintegraloftwo}
	\dashint_{E_{1}}\widehat{e}_{m_{1}}(s_{12})\widehat{e}_{m_{n}}(s_{n1})\mathrm{vol}_{1}
	=-\delta_{m_{1}\geq 1}\cdot \widehat{e}_{m_{n}+m_{1}}(s_{n1}+s_{12})
	+\delta_{m_{n},0}\cdot \widehat{e}_{m_{n}+m_{1}}(s_{n1}+s_{12})\,.
\end{equation}
 		Now set
 		 \[
 		C_{\mathrm{m}}={1\over \widehat{e}_{|\mathrm{m}|}(\sum_{k=1}^{n}s_{k,k+1})}	\dashint_{E_{[n]}}\widehat{e}_{\mathrm{m}}\cdot\bigwedge_{k=1}^{n}\mathrm{vol}_{k}\,.
 		\]
 		Then using \eqref{eqnregularizedintegraloftwo} one has the recursion
 		\[
 		C_{(m_{1},m_{2},\cdots,m_{n})}=(\delta_{m_{n},0}-\delta_{m_{1}\geq 1})
 		C_{(m_{2},\cdots,m_{n}+m_{1})}\,,
 		\]
 		with the values for the initial case $n=2$ determined from 
 		\[
 		\dashint_{E_{[2]}}\widehat{e}_{\mathrm{m}}\cdot\bigwedge_{k=1}^{2}\mathrm{vol}_{k}=\dashint_{E_{2}}(\delta_{m_{2},0}-\delta_{m_{1}\geq 1})\widehat{e}_{m_{2}+m_{1}}(s_{21}+s_{12})\mathrm{vol}_{2}
 		=(\delta_{m_{2},0}-\delta_{m_{1}\geq 1})\widehat{e}_{m_{2}+m_{1}}(s_{21}+s_{12})\,.
 		\]
 If $m_{n}=0$, then $\xi_{\mathrm{loop}}$ reduces to the chain type one $\xi_{\mathrm{chain}}$ discussed above and the result is $\prod_{k=1}^{n-1}\delta_{m_{k},0}$. Otherwise, we have
 		\[
 			C_{\mathrm{m}}=\prod_{k=1}^{n-1}(-\delta_{m_{k}\geq 1})=
 			\prod_{k=1}^{n-1}(\delta_{m_{k},0}-1)\,.
 		\]
 		Combining these results and noting that $\widehat{e}_{0}(z)=1$, this proves the desired claim. 		
 		
 \end{enumerate}
\end{proof}

\begin{rem}\label{remenlargespaceFD}
	
	\begin{enumerate}
		[i).]			
			\item 

	Strictly speaking, partial regularized integration such as the one in \eqref{eqnregularizedintegraloftwo} does not preserve the space  $\mathcal{F}_{\mathsf{D}}$: for example the polar divisor  $s_{n1}+s_{12}=0$ does not necessarily lie in $\mathsf{D}$.
		Thus one needs to enlarge the divisor $\mathsf{D}$ so that 
		it includes all divisors defined by any linear combination of the $s_{ij}$'s.
	This does not cause trouble since all existing results apply to the enlarged divisor straightforwardly.
	
	\item 
	The relation \eqref{eqnregularizedintegraloftwo} can be alternatively derived by reducing an iterated regularized integral of loop type to one of chain type as follows.
	From Lemma \ref{lempartialrelation} (iii), we see that\footnote{The same identity here can also be used to find the residue  of the differential, but now one picks up the $\widehat{e}_{1}$ terms instead of $\widehat{e}_{0}$ terms due to Lemma \ref{lemresidueregularizedintetgralofhatem}.}
	\begin{eqnarray*}
		\widehat{e}_{m_{1}}(s_{12})\widehat{e}_{m_{n}}(s_{n1})
		&=&
		(-1)^{m_{n}}\widehat{e}_{m_{1}+m_{n}}(s_{n1}+s_{12}) \delta_{m_{n},0}
		-\widehat{e}_{m_{1}+m_{n}}(s_{n1}+s_{12}) \\
		&&+\sum_{b=1}^{m_{1}+m_{n}}(-1)^{m_{n}-b}C_{+}(m_{1};m_{n},b)\widehat{e}_{m_{1}+m_{n}-b}(s_{n1}+s_{12})\widehat{e}_{b}(s_{n1})\\
		&&
		+(-1)^{m_{n}}\sum_{a=1}^{m_{1}+m_{n}}C_{-}(m_{1};m_{n},a)\widehat{e}_{m_{1}+m_{n}-a}(s_{n1}+s_{12}) \widehat{e}_{a}(s_{12})
		\,,
	\end{eqnarray*}
	where $C_{+},C_{-}$ are some combinatorial constants.
	The second and third lines on the right hand side above
	give rise to a linear combination of forms, each of which factorizes into
	\[
	\text{non-constant form of chain type}\times \text{form of loop type}\,.
	\]
	The resulting regularized integral is then zero due to the factorization of integrals and the vanishing result on chain type forms.
The regularized integral of the first line gives \eqref{eqnregularizedintegraloftwo} 
as desired.
	
		\end{enumerate}
\end{rem}

In the  calculations in Proposition \ref{propiteratedcalculation}, we need to use
the iterated residue formula \cite[Proposition 2.11]{Li:2022regularized} for the iterated regularized integral.
This requires computing 
the $\bar{\partial}$-primitive\footnote{The result  for the regularized integral is independent of the choice of the primitive by the global residue theorem, as can be seen from the residue formula for the regularized integral.} of the integrand
lying in the space $\mathcal{V}_{\mathsf{D}}$.
Here we have done this using
Lemma \ref{lempartialrelation} (ii) under the generators $\{\widehat{e}_{m}\}_{m\geq 1}$ for the ring $\widehat{\mathfrak{F}}$.
One can alternatively
use
the set of generators $\{\widehat{\mathcal{E}}_{m}^{*}\}_{m\geq 1}$, as is done  in studying properties related to non-meromorphicity (such as holomorphic anomaly in 
\cite{Li:2022regularized}) in which the addition formula for the only non-holomorphic generator $\widehat{\mathcal{E}}^{*}_{1}$ plays the role of  Lemma \ref{lempartialrelation} (iii) here.
The combinatorics involved in finding the $\bar{\partial}$-primitive
for these two sets  of generators are essentially the same, as exhibited in Lemma \ref{lempartialrelation} (ii).
However, for the latter set $\{\widehat{\mathcal{E}}_{m}^{*}\}_{m\geq 1}$, monomials in generators can have  much more complicated residues due to the higher order pole structure present in the generators. 

One encounters the same difficulty when considering iterated residues and iterated regularized integrals of elements in  $\mathcal{F}_{\mathsf{D}}$.

We demonstrate this through the following example.

\begin{ex}
	Consider $f=\widehat{e}_{1}^{r}(z),r\geq 1$. Using the  $\bar{\partial}$-primitive of integrands provided in		Lemma \ref{lempartialrelation} (ii), one has 
	\[
	\dashint_{E}\widehat{e}_{1}^{r}(z)\mathrm{vol}=\mathrm{res}_{\partial}\,
	\left(\sum_{\ell\geq 0}(-1)^{\ell} \widehat{e}_{1}^{r-1-\ell}(z)\widehat{e}_{\ell+2}(z)dz
	\right)\,.
	\]
This is not convenient for actual computations.
	We instead use the alternative $\bar{\partial}$-primitive as  in \cite[Lemma 2.6]{Li:2022regularized}
	\[
	\bar{\partial}\left({1\over r+1}\widehat{e}_{1}^{r+1}(z)dz\right)=\bar{\partial}\widehat{e}_{1}(z)\wedge \widehat{e}_{1}^{r}(z)dz=\pii d\bar{z}\wedge \widehat{e}_{1}^{r}(z)dz
	=-\pii \widehat{e}_{1}^{r}(z)\mathrm{vol}\,
	\]
	to obtain 
	\[
	\dashint_{E}\widehat{e}_{1}^{r}(z)\mathrm{vol}=\mathrm{res}_{\partial}\,
	\left({1\over r+1}\widehat{e}_{1}^{r+1}(z)dz
	\right)\,.
	\]
	Combining \eqref{eqnresidueofemr},
	this then leads to a formula for the desired regularized integral.
	The result is zero when $r$ is odd, as it should be the case
	by the parity reason  explained in Remark \ref{remholomorphicresidueofhatBm}.

	\xxqed
\end{ex}

\subsection{Iterated $A$-cycle integrals and Chen's iterated path integrals}
\label{seciteratedAcycleintegrals}

We now discuss iterated $A$-cycle integrals for elements in $\mathcal{V}_{\mathsf{D}}$
and their relations to Chen's iterated path integrals \cite{Chen:1977}.

\subsubsection{Relation between iterated $A$-cycle integrals and iterated regularized integrals}
\label{subsubsecrelationbetweenAcycleandregularizedintegrals}

We first recall a few results obtained in \cite{Li:2020regularized, Li:2022regularized, Zhou:2023GWgeneratingseries}  on the relation between iterated $A$-cycle integrals and iterated regularized integrals.

First, by the same reasoning in \cite[Lemma 4.12]{Zhou:2023GWgeneratingseries}, when acting on elements in $\mathcal{V}_{\mathsf{D}}$, one has from the simple pole structure that
\begin{equation}\label{eqnaveragedAcycleintegrals}
\mathrm{lim}_{\mathbold{A}=0}
\dashint_{E_{[n]}}={1\over n!}\sum_{\sigma\in \mathfrak{S}([n])}\int_{A_{\sigma([n])}}\,.
\end{equation}
Hence $\dashint_{E_{[n]}}$ can be obtained as the elliptic completion 
of iterated $A$-cycle integrals averaged over the orderings.

Conversely, denote
\begin{equation}
R^{i}_{0}=\int_{A}dz_{i}\,,\quad R^{i}_{j}=\mathrm{res}_{z_{i}=z_{j}} dz_{i}\,.
\end{equation}
When acting on meromorphic functions, these operators satisfy \cite[Lemma A.1]{Li:2022regularized}
\begin{equation}\label{eqncommutatorofRi0Rj0}
	R^{i}_{a} R^{j}_{a}-	R^{j}_{a} R^{i}_{a}=R^{j}_{a}R^{i}_{j}	
	\,,\quad i,j\in [N]\,,~a\in [N]\cup\{0\}\,,
\end{equation}
and
\begin{equation}\label{eqncommutatorofRk0Rij}
R^{k}_{0}R^{i}_{j}=R^{i}_{j}R^{k}_{0}\,,\quad \text{if}~\{i,j\}\cap \{k,0\}=\emptyset\,.
\end{equation}
Let
$\mathcal{R}$ be vector space generated by the operators $R^{i}_{0}, i\in [N]$ and $\mathrm{Lie}(\mathcal{R})$ be its free Lie algebra.
Applying the Poincar\'e-Birkhoof-Witt theorem 
\begin{equation}
	T^{\otimes }\mathcal{R}\cong 
	\mathcal{U}(\mathrm{Lie}(\mathcal{R}))\cong \mathrm{Sym}^{\otimes} \mathrm{Lie}(\mathcal{R})\,,
\end{equation}
one sees that
 any iterated $A$-cycle integral is a symmetric polynomial in the nested Lie brackets of the operators $R^{i}_{0},i\in [N]$.
By \eqref{eqncommutatorofRi0Rj0} and \eqref{eqncommutatorofRk0Rij}, one can show that  \cite[Section 3.4]{Li:2020regularized}  each nested Lie bracket takes the form $R^{i_{m}}_{0}R^{i_{m-1}}_{j_{m-1}}\cdots R^{i_{2}}_{j_{2}}R^{i_{1}}_{j_{1}}$ for some $i_{1},j_{1},\cdots, i_{m-1},j_{m-1},i_{m}\in [N]$.
This implies that  \cite[Theorem 3.9]{Li:2020regularized}  an iterated $A$-cycle integral can be expressed in terms of 
averaged  iterated $A$-cycle integrals of residues.
Applying \eqref{eqnaveragedAcycleintegrals}, it can in turn be expressed in terms of 
holomorphic limits of iterated regularized integrals of residues.

\subsubsection{Chen's iterated path integrals and KZB connection}

Following \cite[Section 1.1]{Calaque:2009}, 
let $\mathfrak{t}_{1,n}$ be the graded Lie algebra with generators $x_{i},y_{i}, i\in [n]$ of degree $1$ and  $t_{ij}, i\neq j\in [n]$, subject to the relations
\begin{eqnarray*}
&t_{ij}=t_{ji}\,,\quad [t_{ij}, t_{ik}+t_{jk}]=0\,,\quad [t_{ij},t_{k\ell}]=0\,,\\
&[x_{i},y_{j}]=t_{ij}\,,\quad [x_{i},x_{j}]=[y_{i},y_{j}]=0\,,\quad [x_{i},y_{i}]=-\sum_{k\neq i}t_{ik}\,,\\
& [x_{i},t_{jk}]=[y_{i}, t_{jk}]=0\,,\quad 
 [x_{i}+x_{j},t_{ij}]= [y_{i}+y_{j},t_{ij}]=0\,,
\end{eqnarray*}
where $i,j,k,\ell$ are distinct. 
The elements $\sum_{i=1}^{n} x_{i}, \sum_{i=1}^{n}y_{i}$ lie in the center of $\mathfrak{t}_{1,n}$
and one denotes by $\bar{\mathfrak{t}}_{1,n}$ the quotient of $\mathfrak{t}_{1,n}$ by the subspace generated by them.

Denote $\mathrm{ad}_{x}=[x,-]\in\mathrm{Der}(\mathfrak{t}_{1,n})$.
The Knizhnik-Zamolodchikov-Bernard (KZB) connection is given \cite[Theorem 3.6]{Calaque:2009} by $\nabla=d-\omega_{\mathrm{KZB}}$, with
\begin{eqnarray}\label{eqnKZBconnection}
\omega_{\mathrm{KZB}}
&=&\sum_{i=1}^{n}dz_{i}\otimes \left(-y_{i}+ \sum_{j\neq i} (S_{c}(z_{ij})-{1\over c}) |_{c=\mathrm{ad}_{x_{i}}} (t_{ij})\right)\,,\quad z_{ij}:=z_{i}-z_{j}\nonumber\\
&=&-\sum_{i=1}^{n}dz_{i}\otimes \sum_{i=1}^{n}y_{i} +\sum_{i=1}^{n}dz_{i}\otimes \left( \sum_{j\neq i} cS_{c}(z_{ij}) |_{c=\mathrm{ad}_{x_{i}}} (y_{j})\right)
\in  A^{1}(\mathrm{Conf}_{n}(E))\otimes \mathfrak{t}_{1,n}\,.
\end{eqnarray}
This gives a  flat connection on the principal bundle  $\mathcal{P}_{1,n}$ on $\mathrm{Conf}_{n}(E)$
whose Lie algebra is  $\mathfrak{t}_{1,n}$.
In fact, one has $d\omega_{\mathrm{KZB}}=\omega_{\mathrm{KZB}}\wedge \omega_{\mathrm{KZB}}=0$.
This connection descends to flat connections on the principal bundle  $\overline{\mathcal{P}}_{1,n}\rightarrow \mathrm{Conf}_{n}(E) $ 
and 
$\overline{\mathcal{P}}_{1,n}\rightarrow \mathrm{Conf}_{n}(E)/E\cong \mathrm{Conf}_{n-1}(E^{\star})   $,
 see \cite[Theorem 3.9]{Calaque:2009}.
 
Chen's $\pi_{1}$ de Rham theorem tells that homotopy invariant integrals on the path spaces of $E$ and $E^{\star}$
are given by  holonomies under the above connections, that is, by iterated path integrals of the corresponding  connections.
See \cite{Chen:1977, Levin:2007, Brown:2011} for more details on this.\\

For simplicity, we 
consider the fiber $F_{n}$ of the relative configuration space 
$\mathrm{Conf}_{n+1}(E)\rightarrow \mathrm{Conf}_{n}(E)$ which is isomorphic to  $E-\{n~\text{points}\}$.
Let $z_{n+1}$ be the coordinate along the fiber, and $z_{j},j\in [n]$ coordinates on the base. 
The KZB connection then induces  a connection $d-\omega_{F_{n}}$ over a bundle $\mathcal{P}_{F_{n}}$ on the fiber space $F_{n}$, with
\begin{eqnarray*}
	\omega_{F_{n}}
	&=&
	-dz_{n+1}\otimes \sum_{i=1}^{n+1}y_{i} +dz_{n+1}\otimes \sum_{j=1}^{n}cS_{c}(z_{n+1}-z_{j})|_{c=\mathrm{ad}_{x_{n+1}}} (y_{j})\\
	&=&
	-dz_{n+1}\otimes \sum_{i=1}^{n+1}y_{i} +dz_{n+1}\otimes cS_{c}(z_{n+1})|_{c=\mathrm{ad}_{x_{n+1}}}(\sum_{i=1}^{n+1}y_{i} )
	\\
	&&
	-dz_{n+1}\otimes 
	cS_{c}(z_{n+1})|_{c=\mathrm{ad}_{x_{n+1}}} y_{n+1}
	+dz_{n+1}\otimes \sum_{j=1}^{n}\left(cS_{c}(z_{n+1}-z_{j}) -cS_{c}(z_{n+1})\right) |_{c=\mathrm{ad}_{x_{n+1}}} (y_{j})\,.
\end{eqnarray*}
Applying the gauge transformation $e^{\mathbold{A}x_{n+1}}$ changes $\omega_{F_{n}}$ to
\begin{eqnarray}\label{eqnomegaFnhat}
	\widehat{\omega}_{F_{n}}&=&\pii \beta_{n+1}\otimes x_{n+1}+e^{\mathbold{A}\,\mathrm{ad}_{x_{n+1}}}\omega_{F_{n}}\\
	&=&
	-dz_{n+1}\otimes e^{\mathbold{A}\mathrm{ad}_{x_{n+1}}}(\sum_{i=1}^{n+1}y_{j})+\pii \beta_{n+1}\otimes x_{n+1}+dz_{n+1}\otimes \sum_{j=1}^{n}c\widehat{S}_{c}(z_{n+1}-z_{j})|_{c=\mathrm{ad}_{x_{n+1}}} (y_{j})\nonumber\,.
\end{eqnarray}
Note the induced flat connection valued in  $\bar{\mathfrak{t}}_{1,n}$ 
is exactly the one given in 
\cite[Section 5.3]{Brown:2011}.

Consider now Chen's iterated path integrals of $\widehat{\omega}_{F_{n}}$ which are related \cite{Brown:2011} to homotopy invariant integrals of the path space of $F_{n}$.
Recall that for a piecewise smooth path $\gamma:[0,1]\rightarrow M$ and smooth $1$-forms $\omega_{k}, k=1,2,\cdots, N$, 
the iterated path integral is defined by \cite{Chen:1977}
\begin{equation}
\int_{\gamma}\omega_{N}\cdots \omega_{1}:=\int_{0\leq t_{1}\leq t_{2}\leq \cdots \leq t_{N}\leq 1} 
 \gamma^{*}\omega_{N}(t_{N})\cdots
 \gamma^{*}\omega_{1}(t_{1})\,.
\end{equation}
We now take $M=F_{n}$, $\omega_{k}=\widehat{\omega}_{F_{n}}, k=1,2,\cdots, N$ and $\gamma$ such that it avoids the singularities of $\widehat{\omega}_{F_{n}}$. This leads to
\begin{equation}\label{eqnKZBiteratedintegral}
\int_{\gamma}\widehat{\omega}_{F_{n}}\cdots \widehat{\omega}_{F_{n}}
=
\int_{0\leq t_{1}\leq t_{2}\leq \cdots \leq t_{N}\leq 1} 
 \gamma^{*}\widehat{\omega}_{F_{n}}(t_{N})\cdots
 \gamma^{*}\widehat{\omega}_{F_{n}}(t_{1})\in F^{N}\mathcal{U}(\bar{\mathfrak{t}}_{1,n})\,,
\end{equation}
where $F^{\bullet}$ is the increasing filtration on the universal enveloping algebra $\mathcal{U}(\bar{\mathfrak{t}}_{1,n})$ defined by the length of words.
Using \eqref{eqnomegaFnhat} and Definition \ref{dfnSzegokernel}, we write
\[
 \gamma^{*}\widehat{\omega}_{F_{n}}(t) =
 \pii \beta(t)\otimes x_{n+1}
 + \sum_{j=1}^{n}\sum_{m\geq 0} \widehat{e}_{m}(t-z_{j})dt\otimes  \mathrm{ad}_{x_{n+1}}^{m}(y_{j})\in A^{1}(F_{n})\otimes \bar{\mathfrak{t}}_{1,n}\,.
\]
The iterated path integral 
\eqref{eqnKZBiteratedintegral} then gets simplified into
\begin{equation}\label{eqnKZBiteratedintegralexpansioninwords}
\int_{\gamma}\widehat{\omega}_{F_{n}}\cdots \widehat{\omega}_{F_{n}}=
\sum_{m_{k}\geq 0 }\sum_{j_{k}\geq 0}
\int_{\gamma}
\prod_{k=1}^{N} \widehat{e}_{m_{k}}(t_{k}-z_{j_{k}})dt_{k}\bigotimes \otimes_{k=1}^{N} \mathrm{ad}_{x_{n+1}}^{m_{k}}(y_{j_{k}})\in \mathcal{U}(\bar{\mathfrak{t}}_{1,n})\,,
\end{equation}
where the convention for $j_{k}=0$ is that
$\widehat{e}_{m_{k}}(t_{k}-z_{j_{k}})dt= \delta_{m_{k},0}\pii \beta$ and $  \mathrm{ad}_{x_{n+1}}^{m_{k}}(y_{j_{k}})= \delta_{m_{k},0}x_{n+1}$.


For the aforementioned $\gamma$,  one can also consider the ordinary integral (invoke the Fubini theorem)
\begin{equation}\label{eqnorderedAcycleintegralofKZB}
\int_{\gamma^{N}}\prod_{k=1}^{N} \widehat{e}_{m_{k}}(t_{k}-z_{j_{k}})dt_{k}
=\prod_{k=1}^{N}\int_{\gamma} \widehat{e}_{m_{k}}(t_{k}-z_{j_{k}})dt_{k}
\,.
\end{equation}
This corresponds to the iterated $A$-cycle integral with the choice 
$\varepsilon_{i}=\varepsilon_{1},i\geq 2$ in 
\eqref{eqnconventionforiteratedintegrals}, but with the integrand given by a smooth form instead of a meromorphic form. By working with the polar coordinates near the singularity one can see from Lemma \ref{lempolarparts} that the form $d(\widehat{e}_{m}(z)dz)\in A^{2}_{E}( 0)$ is in fact integrable on $E$. Applying Stokes theorem and a limiting argument to each factor in \eqref{eqnorderedAcycleintegralofKZB}, we see that the iterated $A$-cycle integral is invariant under small deformations of $\varepsilon_{1}$.\\

 We have the following result regarding the Chen iterated integrals in \eqref{eqnKZBiteratedintegralexpansioninwords}.
 \begin{lem}\label{lemKZBiteratedintegralanditeratedintegral}
 Let the notation be as above.
 Take $\gamma$ to be a (non-unique) $A$-cycle on $E$ avoiding the pole divisor of $\widehat{\omega}_{F_{n}}$
and consider the image of 
 $\int_{\gamma}\widehat{\omega}_{F_{n}}\cdots \widehat{\omega}_{F_{n}}$ in $\mathrm{gr}^{N}_{F} \mathcal{U}(\bar{\mathfrak{t}}_{1,n})$.
 Then the coefficient of the image of $ \otimes_{k=1}^{N} \mathrm{ad}_{x_{n+1}}^{m_{k}}(y_{j_{k}})$
is given by the iterated $A$-cycle integral $\int_{\gamma^{N}}\prod_{k=1}^{N} \widehat{e}_{m_{k}}(t_{k}-z_{j_{k}})dt_{k}$.
 \end{lem}
 \begin{proof}
 	A cycle $\gamma$ on $E$ avoiding the pole divisor of $\widehat{\omega}_{F_{n}}$
 	gives rise to a cycle on the space $F_{n}$.
 According to
 \eqref{eqnKZBiteratedintegralexpansioninwords},   
if we
take the path $\gamma$ be an $A$-cycle,  integrals involving $j_{k}=0$ terms  are eliminated since $\gamma^{*}\beta=0$. 
That is, the expression \eqref{eqnKZBiteratedintegralexpansioninwords} gives a genuine iterated integral along the path $\gamma$.

 Let us focus on the image of  $\int_{\gamma}\widehat{\omega}_{F_{n}}\cdots \widehat{\omega}_{F_{n}}$  
 in $\mathrm{gr}^{N}_{F} \mathcal{U}(\bar{\mathfrak{t}}_{1,n})$,
  we can pretend that $XY=YX$ for $X,Y\in  \bar{\mathfrak{t}}_{1,n}$ since $(XY-YX)F^{N-2} \mathcal{U}(\bar{\mathfrak{t}}_{1,n})=0\in \mathrm{gr}^{N}_{F} \mathcal{U}(\bar{\mathfrak{t}}_{1,n})$.
 The coefficient of the image of $ \otimes_{k=1}^{N} \mathrm{ad}_{x_{n+1}}^{m_{k}}(y_{j_{k}})$  in $\mathrm{gr}^{N}_{F} \mathcal{U}(\bar{\mathfrak{t}}_{1,n})$ then becomes, as claimed,
 \begin{eqnarray*}
 &&\sum_{\sigma\in \mathfrak{S}_{N}}
 \int_{0\leq t_{1}\leq t_{2}\leq \cdots\leq t_{N}\leq 1}
\prod_{k=1}^{N} \widehat{e}_{m_{\sigma(k)}}(t_{k}-z_{j_{\sigma(k)}})dt_{k}\\
&=&\int_{[0,1]^{N}  }
\prod_{k=1}^{N} \widehat{e}_{m_{k}}(t_{k}-z_{j_{k}})dt_{k}\\
&=&\int_{\gamma^{N}}\prod_{k=1}^{N} \widehat{e}_{m_{k}}(t_{k}-z_{j_{k}})dt_{k}\,.
 \end{eqnarray*}
 \end{proof}
 
Applying the discussions in Section \ref{subsubsecrelationbetweenAcycleandregularizedintegrals}, the iterated $A$-cycle integrals in Lemma \ref{lemKZBiteratedintegralanditeratedintegral} above can be further expressed in terms of iterated regularized integrals 
computed in Section \ref{seccomputationsoniteratedresiduesandintegrals}.

We can also take the path $\gamma$ to be the $B$-cycle, 
the same discussions lead to iterated $B$-cycle integrals. They are related to iterated $A$-cycle integrals in a nice way.
Indeed, 
by the reciprocity law applied to $\widehat{e}_{m}(z)dz,dz$, one  has (using Lemma \ref{lemresidueregularizedintetgralofhatem})
\[
\int_{B}\widehat{e}_{m}(z)dz=\tau \int_{A}\widehat{e}_{m}(z)dz+\mathrm{res}_{\partial}(z\widehat{e}_{m}(z)dz)=\tau \int_{A}\widehat{e}_{m}(z)dz\,.
\]

It seems interesting to find the image in $\mathrm{gr}_{F}^{\bullet}\, \mathcal{U}(\bar{\mathfrak{t}}_{1,n})$ besides the top component above.
In light of the relation \cite[Theorem 1.2]{Zhou:2023GWgeneratingseries} between iterated $A$-cycle integrals and GW invariants of elliptic curves, it would also be interesting to find realizations of more general iterated $A$-cycle integrals 
considered in  Section \ref{seccomputationsoniteratedresiduesandintegrals} in terms of Chen's iterated path integrals.

\section{Gromov-Witten generating series of elliptic curves}
\label{secGW}

In this section, we prove our main result Theorem \ref{mainthm1intro}
based on the results obtained in Section \ref{seciteratedintegrals}.
We also discuss some  combinatorial properties of the generating series of $\{\dashint_{E^{n}}\omega_{n}\}_{n\geq 1}$.

\subsection{Proof of the main theorem}\label{secproofofmainthm1intro}

Using the notation and conventions in earlier sections (e.g., Definition \ref{dfnringofquasiellipticfunctions}, Convention \eqref{eqnconventionforiteratedintegrals}),  Theorem 
\ref{mainthm1intro} is recalled as follows.

\begin{thm}\label{mainthm1}

	Let 
	\begin{equation}\label{eqnomegaC2n}
		\varpi_{n}=
		{\theta (\sum_{i=1}^{n} w_{i}) \over \prod_{i=1}^{n} \theta (w_{i})}\cdot
		{	\prod_{i<j }\theta (z_{i}+w_{i}-z_{j}-w_{j})\theta (z_{i}-z_{j}) \over  \prod_{i<j}
			\theta (z_{i}+w_{i}-z_{j}) \theta (z_{i}-w_{j}-z_{j})}
		\,.
	\end{equation}
	Then the regularized integral 
	 $\dashint_{E_{[n]}}\varpi_{n}$ satisfies
	\begin{equation}\label{eqnTngwpureweightstructure}
		\dashint_{E_{[n]}}\varpi_{n}=
		\sum_{j\in [n]}\sum_{\pi\in \Pi_{[n]\setminus \{j\}}}
	\prod_{k=1}^{\ell} (|\pi_{k}|-1)!\cdot \widehat{e}_{|\pi_{k}|}(\sum_{i\in \pi_{k}}w_{i})\,.
	\end{equation}
	
\end{thm}

From \eqref{eqnautomorphyoftheta}, one can check that the function $\varpi_{n}$
is indeed elliptic in both $z_{i},w_{i}$, where the $z_{i}$'s are regarded as variables on $E_{[n]}$, while the $w_{i}$'s as nonzero parameters.

Denote the zero and polar divisor of $\varpi_{n}$ by
$\mathsf{D}_{0},\mathsf{D}_{\infty}$ respectively.
The divisor $\mathsf{D}_{\infty}$ is a smooth hypersurface arrangement divisor (in fact a  simple normal crossing divisor for generic values of the $w_{i}$ parameters). In particular, this fits in the setting of
cohomological regularized integral in 
\cite{Zhou:2023cohomologicalpairings}, according to
which the regularized integral $	\dashint_{E_{[n]}}\varpi_{n}$
above is in fact purely cohomological.\\

In what follows, we shall evaluate 	$\dashint_{E_{[n]}}\varpi_{n}$ using the results established in Section \ref{seciteratedintegrals}.
To reduce the difficulty caused by 
the complexity of the polar divisor $\mathsf{D}_{\infty}$, we first need to split $\varpi_{n}
$ into a sum of forms with simpler polar divisors in a way similar to the one provided by Lemma 	\ref{lempartialrelation} (iii).
This is possible thanks to the following version of Fay's multi-secant identity.
 
\begin{lem}\label{lemFrobenius-Stickelbergeralmostelliptic}
Let the notation be as above. Then one has 
		\begin{equation} \label{eqnFrobenius-Stickelbergeralmostelliptic}
		\varpi_{n}
		=\det\,
		\begin{pmatrix}
			0 & 1 & \cdots & 1\\
			-1 & {\widehat{\theta}'\over\widehat{\theta}}(w_{1}+z_{1}-z_{1}) &  \cdots & {\widehat{\theta}'\over \widehat{\theta}}(w_{1}+z_{1}-z_{n})\\
			\vdots &\vdots& {\widehat{\theta}'\over\widehat{\theta}}(w_{i}+z_{i}-z_{j})& \vdots\\
			-1 &  {\widehat{\theta}'\over \widehat{\theta}}(w_{n}+z_{n}-z_{1}) &\cdots&  {\widehat{\theta}'\over \widehat{\theta}}(w_{n}+z_{n}-z_{n})\\
		\end{pmatrix}
		\,.
	\end{equation}

	\end{lem}
\begin{proof}
	The 
Frobenius-Stickelberger formula/Fay's multi-secant identity  gives
(see e.g., \cite[Lemma 3.3]{Zhou:2023GWgeneratingseries})
	\begin{equation}\label{eqnFrobenius-Stickelbergerquasielliptic}
		\varpi_{n}
		=\det\,
		\begin{pmatrix}
			0 & 1 & \cdots & 1\\
			-1 & {\theta'\over \theta}(w_{1}+z_{1}-z_{1}) &  \cdots & {\theta'\over \theta}(w_{1}+z_{1}-z_{n})\\
			\vdots &\vdots& {\theta'\over\theta}(w_{i}+z_{i}-z_{j})& \vdots\\
			-1 &  {\theta'\over \theta}(w_{n}+z_{n}-z_{n}) &\cdots&  {\theta'\over \theta}(w_{n}+z_{n}-z_{n})
		\end{pmatrix}\,.
	\end{equation}	
	The right hand side in \eqref{eqnFrobenius-Stickelbergeralmostelliptic} is almost-elliptic in $z_{i},w_{i}$, whose holomorphic limit gives the right hand side of \eqref{eqnFrobenius-Stickelbergerquasielliptic} which is already almost-elliptic in $z_{i},w_{i}$ since $\varpi_{n}$ is so.\footnote{One can also check the almost-ellipticity of the right hand side of \eqref{eqnFrobenius-Stickelbergerquasielliptic} directly, using the relations in Definition
		\ref{dfnhatthetazeta}, the automorphy behavior \eqref{eqne1automorphy},  and the linearity of determinant in row and column vectors.}
	Therefore, they are identical.
	
\end{proof}

Note that although
the entire determinant in \eqref{eqnFrobenius-Stickelbergerquasielliptic} above have trivial automorphy in the $z_{i},w_{i}$'s,
the individual summands in its expansion are not so.
This is why we passed from ${\theta'/ \theta}$ to its elliptic completion ${\widehat{\theta}'/\widehat{\theta}}$, and later on work with the expansion
of  \eqref{eqnFrobenius-Stickelbergeralmostelliptic}.
Explicitly, using \eqref{eqnhatthetazeta} and \eqref{eqnfirstfewhatems}
we have
\begin{equation}\label{eqndfnoftijforGW}
	{\widehat{\theta}'\over \widehat{\theta}}
	=\widehat{e}_{1}(s_{ij})\,,\quad 
s_{ij}:=w_{i}+z_{i}-z_{j}\,.
\end{equation}

\begin{proof}[Proof of Theorem \ref{mainthm1}]
Denote the matrix in Lemma
\ref{lemFrobenius-Stickelbergeralmostelliptic} by $\mathcal{Z}$ whose rows and columns we label by $\{0\}\cup [n]$.
By the expansion of determinant, one has 
\begin{equation}\label{eqndeterminantpermutationexpansion}
\varpi_{n}=\det \mathcal{Z}=\sum_{\rho\in \mathfrak{S}_{\{0\}\cup [n]}}(-1)^{\mathrm{sign}(\rho)}\prod_{k\in \{0\}\cup [n]}\mathcal{Z}_{k\rho(k)}\,.
\end{equation}
Denote 
\begin{equation}\label{eqnxirhoform}
\xi_{\rho}=(-1)^{\mathrm{sign}(\rho)}\prod_{k\in \{0\}\cup [n]}\mathcal{Z}_{k\rho(k)}\cdot \bigwedge_{k=1}^{n}\mathrm{vol}_{k}\,.
\end{equation}
Representing any permutation $\rho$ by its standard cycle decomposition,
this leads to factorization of the corresponding iterated regularized integral $\dashint_{E_{[n]}}\xi_{\rho}$. 
Focus on the cycle that contains $0$. It is either the trivial one $(0)$,
or one of the form $(0j_{0}j_{1}\cdots j_{r-1} j_{r})$ with length $r+2,r\geq 0$.
For the former, the contribution to $\dashint_{E_{[n]}}\xi_{\rho}$ gives zero since $\mathcal{Z}_{00}=0$ in \eqref{eqnxirhoform}.
For the latter, it gives
\begin{equation}\label{eqnomissionincycleofrho}
 (-1)^{r+2-1}
	\mathcal{Z}_{0 j_0}\mathcal{Z}_{j_0 j_1}\cdots \mathcal{Z}_{j_{r-1}j_{r}}
	\mathcal{Z}_{j_{r}, j_{r+1}=0}
=(-1)^{r+1-1} 
	  \mathcal{Z}_{j_0 j_1}\cdots \mathcal{Z}_{j_{r-1}j_{r}}\,.
\end{equation}
If $j_r\neq j_0$, then the above cycle $(0j_{0}j_{1}\cdots j_{r-1} j_{r})$ yields a form of type  $\xi_{\mathrm{chain}}$ 
and thus the
 corresponding regularized integral  $\dashint_{E_{[n]}}\xi_{\rho}$ is zero by Proposition \ref{propiteratedcalculation} (ii).
It follows that the nontrivial contributions $\dashint_{E_{[n]}}\xi_{\rho}$ to the regularized integral 
$\dashint_{E_{[n]}}\varpi_{n}$ arise from those $\rho$
satisfying $j_{r}=j_{0}$, that is, with $(0j_{0}j_{1}\cdots j_{r-1} j_{r})=(0j_{0})$.
This gives
\begin{equation}\label{eqnsumoversigma}
\dashint_{E_{[n]}}\varpi_{n}=\dashint_{E_{[n]}}
\sum_{j_0=1}^{n}
\sum_{\sigma\in \mathfrak{S}_{[n]\setminus \{j_{0}\}}} (-1)^{\sigma}\mathcal{Z}_{1\sigma_{1}}\mathcal{Z}_{2\sigma_{2}}\cdots \check{\mathcal{Z}}_{j_{0}j_{0}}\cdots \mathcal{Z}_{n\sigma_{n}}\cdot \bigwedge_{k=1}^{n}\mathrm{vol}_{k}\,,
\end{equation}
where $\check{\mathcal{Z}}_{j_{0}j_{0}}$ stands for the omission of 
$\mathcal{Z}_{j_{0}j_{0}}$  due to 
\eqref{eqnomissionincycleofrho}.

To compute \eqref{eqnsumoversigma}, again we use the standard cycle decomposition $\sigma=\sigma_{1}\cdots \sigma_{\ell}$ for any permutation $\sigma\in \mathfrak{S}_{[n]\setminus \{j_0\}}$.
Observe that 
the iterated regularized integral corresponding to a cycle $\sigma_{k}$ only depends on 
its underlying set $\pi_{k}$, by Proposition \ref{propiteratedcalculation} (ii) and
the structure of the polar divisor given in \eqref{eqndfnoftijforGW}, with
\begin{equation}\label{eqnreductionfromsigmaktopik}
\sum_{i\in \pi_{k}}s_{i,\sigma_{i}}=\sum_{i\in \pi_{k}} w_{i}\,.
\end{equation}
It follows that the sum in \eqref{eqnsumoversigma} becomes a double sum, with the outer sum being a sum over partitions  $\pi=\{\pi_{1},\cdots,\pi_{\ell}\} \in \Pi_{[n]\setminus \{j_{0}\}}$, and the inner sum over permutations within the blocks of the partitions.
The number of permutations $\sigma \in \mathfrak{S}_{[n]\setminus \{j_{0}\}}$ 
whose standard cycle decompositions share the same underlying partition $\pi \in \Pi_{[n]\setminus \{j_{0}\}}$ is $\prod_{k=1}^{\ell}(|\pi_{k}|-1)!$\,.
Since the sign contributing to $(-1)^{\sigma}$ of each block $\pi_{k}$
is $(-1)^{|\pi_{k}|-1}$,
applying Proposition
\ref{propiteratedcalculation} (ii) to each such block simplifies \eqref{eqnsumoversigma} into 
\begin{eqnarray*}
	\dashint_{E_{[n]}}\varpi_{n}=
	\sum_{j_0=1}^{n}
	\sum_{\pi\in \Pi_{[n]\setminus \{j_{0}\}}}
	\prod_{k=1}^{\ell} (-1)^{|\pi_{k}|-1}\cdot (|\pi_{k}|-1)!\cdot  \widehat{e}_{|\pi_{k}|}(\sum_{i\in \pi_{k}}s_{i,\sigma_{i}})\cdot 	\left(	\prod_{i\in \pi_{k}}\delta_{1,0}-
	\prod_{i\in \pi_{k}}(\delta_{1,0}-1)\right)\,.
	\end{eqnarray*}
Simplifying further using \eqref{eqnreductionfromsigmaktopik}, one obtains the desired claim.
 \end{proof}

\subsection{Generating series of $n$-point functions}

The results we have established in previous sections,  such as Lemma 
\ref{lemFrobenius-Stickelbergeralmostelliptic}, provide new perspectives in studying finer combinatorial properties of the GW generating series.
In the rest of the work, we initiate these studies by  working with the generating series of $\{\widehat{T}_{n}\}_{n\geq 1}$.
We hope to connect 
our results to constructions on infinite Grassmannians and integrable hierarchies 
 along these lines in further investigations.\\

We first introduce some notation.

\begin{dfn}\label{eqndfnZinfty}
	Let the notation be as before. Denote
\begin{eqnarray}
\widehat{Z}_{n}&=&
\begin{pmatrix}
		 {\widehat{\theta}'\over \widehat{\theta}}(w_{1}+z_{1}-z_{1}) &  \cdots & {\widehat{\theta}'\over \widehat{\theta}}(w_{1}+z_{1}-z_{n})\\
		\vdots& \cdots& \vdots\\
		 {\widehat{\theta}'\over \widehat{\theta}}(w_{n}+z_{n}-z_{1}) &\cdots&  {\widehat{\theta}'\over \widehat{\theta}}(w_{n}+z_{n}-z_{n})\\
		\end{pmatrix}\,,\quad \delta_{n}=
\begin{pmatrix}
			 \varepsilon_1 \mathrm{vol}_{1} &  \cdots &0\\
		\vdots & \cdots& \vdots\\
		 0 &\cdots& \varepsilon_n\mathrm{vol}_{n}\\
		\end{pmatrix}\,,\\
		\widehat{\mathcal{Z}}_{n}&=&
\begin{pmatrix}
			0 & 1 & \cdots & 1\\
			-1 & {\widehat{\theta}'\over \widehat{\theta}}(w_{1}+z_{1}-z_{1}) &  \cdots & {\widehat{\theta}'\over \widehat{\theta}}(w_{1}+z_{1}-z_{n})\\
			\vdots &\vdots& \cdots& \vdots\\
			-1 &  {\widehat{\theta}'\over \widehat{\theta}}(w_{n}+z_{n}-z_{1}) &\cdots&  {\widehat{\theta}'\over \widehat{\theta}}(w_{n}+z_{n}-z_{n})\\
		\end{pmatrix}\,,\quad \Delta_{n}=
\begin{pmatrix}
			\varepsilon_0 & 0 & \cdots & 0\\
			0 & \varepsilon_1 \mathrm{vol}_{1} &  \cdots &0\\
			\vdots &\vdots& \cdots& \vdots\\
			0& 0 &\cdots& \varepsilon_n \mathrm{vol}_{n}\\
		\end{pmatrix}\,,\nonumber
		\end{eqnarray}
		where $\varepsilon_{0},\varepsilon_{k}, 1\leq k\leq n$ are indeterminants.
				If needed, we could add in an auxiliary curve $E_{0}$ and replace $\varepsilon_{0}$ by  
		$\varepsilon_{0}\mathrm{vol}_{0}$ correspondingly. 
		Similarly, we define $Z_{n}, \mathcal{Z}_{n}$ to be the corresponding holomorphic limit versions obtained by replacing $\widehat{\theta}$ by $\theta$ everywhere.
			
\end{dfn}	
		
	By Lemma \ref{lemFrobenius-Stickelbergeralmostelliptic} and  results in Section \ref{secproofofmainthm1intro}, we have
		\begin{equation}\label{eqndeterminantofquasimatrix}
		\det  \,\mathcal{Z}_{n}= \det \,\widehat{\mathcal{Z}}_{n}\,,\quad
		\prod_{k=0}^{n}\varepsilon_{k}\cdot 
		\widehat{T}_{[n]}=\dashint_{E_{[n]}}\det \,(\Delta_{n}\widehat{\mathcal{Z}}_{n})\,.
		\end{equation}

Constructions using the matrix $Z_{n}$ lead to  results parallel to those
obtained using $\widehat{\mathcal{Z}}_{n}$.
In fact, many properties are easier to understand by putting them together.

Before proceeding, we need more definitions.

\begin{dfn}\label{dfnofGn}
	Let the notation and conventions be as above.
	Define
	\begin{equation}
	\prod_{k=1}^{n}\varepsilon_{k}\cdot \widehat{G}_{n}:=\dashint_{E_{[n]}}\det\, (\delta_{n}\widehat{Z}_{n})\,.
	\end{equation}
		We extend the definition of $\widehat{T}_{n},\widehat{G}_{n}$ to any set $S\subseteq \mathbb{N}_{+}$, and set by convention
	\begin{equation}
	\widehat{T}_{\emptyset}=\widehat{T}_{0}=1\,,\quad \widehat{G}_{\emptyset}=\widehat{G}_{0}=1\,.
	\end{equation}		
\end{dfn}

\begin{lem}\label{lemGnhatassumoverpartition}
	Let the notation and conventions be as above.
	Then one has
	\begin{equation}\label{eqnGnsumoverpartition}
		\widehat{G}_{[n]}=
		\sum_{\pi\in \Pi_{[n]}}
		{\widehat{\mathbold{B}}_{{\pi}}\over |\pi|}
		\,,
	\end{equation}
where the right hand side is defined as in \eqref{eqnTngwpureweightstructureintro} of Theorem \ref{mainthm1intro}.
\end{lem}
\begin{proof}
	This follows by repeating the proof of Theorem \ref{mainthm1}.
\end{proof}

Note that unlike \eqref{eqndeterminantofquasimatrix}, one has
$\det\, Z_{n}\neq \det \,\widehat{Z}_{n}$, and only 
the regularized integral of $\det\,\widehat{Z}_{n}$ makes sense since $\det \,Z_{n}$
does not define a function on $E_{[n]}$ due to its  quasi-ellipticity.

\begin{dfn}\label{dfntopology}
Define
		\begin{equation}
		\mathcal{F}_{\mathsf{D}}[[\varepsilon_0,\varepsilon_{1}, \cdots , \varepsilon_{n},\cdots]]=\lim_{\leftarrow} \mathcal{F}_{\mathsf{D}}[[\varepsilon_0,\varepsilon_{1}, \cdots , \varepsilon_{n}]]\,,
		\end{equation}
		where the morphisms in the projective system are given by 
		\[
		\mathcal{F}_{\mathsf{D}}[[\varepsilon_0,\varepsilon_{1}, \cdots , \varepsilon_{n}, \varepsilon_{n+1}]]\rightarrow 
		\mathcal{F}_{\mathsf{D}}[[\varepsilon_0,\varepsilon_{1}, \cdots , \varepsilon_{n}]]\,,\quad 
		\varepsilon_{k}\mapsto \varepsilon_{k}\,, k\leq n\,,~ \varepsilon_{n+1}\mapsto 0\,.
		\]
		It follows that 
		\[
		\mathcal{F}_{\mathsf{D}}[[\varepsilon_0,\varepsilon_{1}, \cdots , \varepsilon_{n},\cdots]]=\lim_{\leftarrow}\left(	\mathcal{F}_{\mathsf{D}}[\varepsilon_0,\varepsilon_{1}, \cdots , \varepsilon_{n}]/(\varepsilon_0,\varepsilon_{1},\cdots,\varepsilon_{n})^{n+1}\right)\,.\]
		Let $\mathcal{J}$ be the ideal of 
$\mathcal{F}_{\mathsf{D}}[[\varepsilon_0,\varepsilon_{1}, \cdots , \varepsilon_{n},\cdots]]$ generated 
	by $\varepsilon_{1}^2\,,\cdots, \varepsilon_{n}^2\,,\cdots$.
\end{dfn}

\begin{dfn}\label{dfnofgeneratingseriesofTn}
		Let $\mathbb{1}_{\infty}, \widehat{\mathcal{Z}}_{\infty},\widehat{Z}_{\infty},\Delta_{\infty},\delta_{\infty}$ be the limit matrices of 
	$\{\mathrm{id}_{n}\}_{n\geq 1}$, $\{\widehat{\mathcal{Z}}_{n}\}_{n\geq 1}$, $\{\widehat{Z}_{n}\}_{n\geq 1}$,  $\{\Delta_{n}\}_{n\geq 1}$, $\{\delta_{n}\}_{n\geq 1}$, respectively.
	Define 
	\begin{equation}
	\det \, (
	\mathbb{1}_{\infty}+\Delta_{\infty}\widehat{\mathcal{Z}}_{\infty})
	\,,\quad 
	  \det \, (
	\mathbb{1}_{\infty}+\delta_{\infty}\widehat{Z}_{\infty})\in \mathcal{F}_{\mathsf{D}}[[\varepsilon_0,\varepsilon_{1}, \cdots , \varepsilon_{n},\cdots]]
	\end{equation}
	to
	be the summation of the principal minors of $\Delta_{\infty}\widehat{\mathcal{Z}}_{\infty}, \delta_{\infty}\widehat{Z}_{\infty}$, respectively.
		
	Let $\varepsilon_{S}:=\prod_{k\in S}\varepsilon_{k}$.
	Define 
		\begin{equation}
	\widehat{\mathcal{T}}_{[n]}(\varepsilon)=\sum_{S\subseteq  [n]}\varepsilon_{S}\widehat{T}_{S}
	=1+\varepsilon_{0}(\varepsilon_{1}+\varepsilon_{2}+\cdots+\varepsilon_{n})+\cdots\,,\quad 
	\widehat{\mathcal{G}}_{[n]}(\varepsilon)=\sum_{S\subseteq [n]}\varepsilon_{S}\widehat{G}_{S}=1+\cdots\,,
	\end{equation}
	and  correspondingly 
	the following elements in the ring $\mathcal{F}_{\mathsf{D}}[[\varepsilon_0,\varepsilon_{1}, \cdots , \varepsilon_{n},\cdots]]$
		\begin{equation}
	\widehat{\mathcal{T}}(\varepsilon)=\sum_{S\subseteq \mathbb{N}_{+}}\varepsilon_{S}\widehat{T}_{S}
	=1+\varepsilon_{0}(\varepsilon_{1}+\varepsilon_{2}+\cdots)+\cdots\,,\quad 
	\widehat{\mathcal{G}}(\varepsilon)=\sum_{S\subseteq \mathbb{N}_{+}}\varepsilon_{S}\widehat{G}_{S}=1+\cdots\,.
		\end{equation}

\end{dfn}

\begin{prop}\label{propHTGrelation}
	Let the notation and conventions be as above.
	\begin{enumerate}[i).]
		
		\item One has 
		\[				
				\widehat{\mathcal{G}}(\varepsilon)=	\dashint_{E_{\infty}}\det \, (
				\mathbb{1}_{\infty}+\delta_{\infty}\widehat{Z}_{\infty})\,,			
		\]
		where $\dashint_{E_{\infty}}$ is understood to be the identity operator
		when acting on a constant form.
		\item One has 
		\begin{equation*}\label{eqnrelationbetweenTandG}
			\widehat{\mathcal{T}}(\varepsilon)-1-\varepsilon_{0}(\varepsilon_{1}+\varepsilon_{2}+\cdots)
			=(\widehat{\mathcal{T}}(\varepsilon)-1)(\widehat{\mathcal{G}}(\varepsilon)-1)\quad \mathrm{mod}~\mathcal{J}\,.
		\end{equation*}
		
		\item 
				Define  
		\[
		\widehat{\mathcal{H}}(\varepsilon)=\dashint_{E_{\infty}}\det \, (
		\mathbb{1}_{\infty}+\Delta_{\infty}\widehat{\mathcal{Z}}_{\infty})\,.
		\]
		Then one has	
		\[ 
		\widehat{\mathcal{H}}(\varepsilon)=\varepsilon_{0}\cdot (\widehat{\mathcal{T}}(\varepsilon)-1)+\widehat{\mathcal{G}}(\varepsilon)\,.
		\]		
	\end{enumerate}
\end{prop}

\begin{proof}

	\begin{enumerate}[i).]
		\item  This follows immediately from Definition \ref{dfnofgeneratingseriesofTn} and the proof of Theorem \ref{mainthm1}.
		
		\item 
			Consider a square matrix $M$ 
		whose rows and columns are labeled by $S\subseteq \mathbb{N}$. For any subset $I\subseteq S$, let 
		$M_{I}$ be the submatrix of $M$ formed by collecting the rows and columns with labels from $I$. For a subset $I\subseteq S$, set $I^{c}=S\setminus I$.
		
		By the proof of Theorem \ref{mainthm1}, we have
		\begin{eqnarray*}\label{eqnCramerfordetZn}
			\dashint_{E_{[n]}}
			\det \,(\Delta_{n}\widehat{\mathcal{Z}}_{n})&=&	\dashint_{E_{[n]}}\sum_{\emptyset \subsetneq I\subsetneq [n] }\det\, (\Delta_{I} \widehat{\mathcal{Z}}_{I})\cdot \det\, (\delta_{I^{c}} \widehat{Z}_{I^{c}})\\
			&=&\sum_{\emptyset \subsetneq I\subsetneq [n] }\dashint_{E_{I}}\det\, (\Delta_{I} \widehat{\mathcal{Z}}_{I})\cdot \dashint_{E_{I^{c}}}\det\, (\delta_{I^{c}} \widehat{Z}_{I^{c}})\,,\quad n\geq 2\,.
		\end{eqnarray*}
	Summing over $n$ leads to the desired claim, where the need of modulo $\mathcal{J}$
	arises from the relation $I\cap I^{c}=\emptyset$.
		\item 
		By Definition \ref{dfnofgeneratingseriesofTn}, we have
		\[\varepsilon_{0}^{1}\cdot [\varepsilon_{0}^{1}]
		\widehat{\mathcal{H}}(\varepsilon)=\widehat{\mathcal{T}}(\varepsilon)-1\,,\quad 
		[\varepsilon_{0}^{0}]
		\widehat{\mathcal{H}}(\varepsilon)=\dashint_{E_{\infty}}\det \, (\mathbb{1}_{\infty}+\delta_{\infty}\widehat{Z}_{\infty})=\widehat{\mathcal{G}}(\varepsilon)\,.
		\]
		Combining these, 
		one obtains the desired claim.

		\end{enumerate}
	\end{proof}

\subsection{Deformation in Grassmannian}

Proposition \ref{propHTGrelation} 
above tells that any of 
$\widehat{\mathcal{T}},\widehat{\mathcal{G}},\widehat{\mathcal{H}}$ determine the rest two.
Therefore, it suffices to study the properties of one of them.
In this work we prefer to work with $\widehat{\mathcal{G}}$ as 
\eqref{eqnGnsumoverpartition} in Lemma \ref{lemGnhatassumoverpartition} seems to be simpler than \eqref{eqnTngwpureweightstructureintro} in Theorem \ref{mainthm1intro}.
Furthermore, $\widehat{\mathcal{G}}$ admits a natural geometric formulation 
similar to the $\tau$-function in the study of KP hierarchy \cite{DJKM81, DJKM83, Segal:1985}, as we now explain.\\

Let $J$ be the Jacobian of $E$ and $\mathsf{Pic}^{(0)}(J)\cong J$ be
its Picard variety. Denote by $m$ the addition map
\[
m: J\times \mathsf{Pic}^{(0)}(J)\rightarrow J\,.
\]
The Poincar\'e bundle on $J\times \mathsf{Pic}^{(0)}(J)$ satisfies
\begin{equation}
\mathcal{P}\cong \mathcal{O}_{J\times \mathsf{Pic}^{(0)}(J)}(-[J\times 0]-[0\times  \mathsf{Pic}^{(0)}(J)]+[\mathrm{ker}\,m])\,.
\end{equation}
We consider constructions on $\mathcal{P}(J\times 0)|_{2(J\times 0)}$ which fits in
the decomposition sequence 
\begin{equation}\label{eqndecompositionsequence}
0\rightarrow 
\mathcal{P}|_{J\times 0}
\rightarrow 
\mathcal{P}(J\times 0)|_{2(J\times 0)}\rightarrow \mathcal{P}(J\times 0)|_{J\times 0}\rightarrow 0\,,
\end{equation}
where 
\begin{equation}\label{eqn0thordertriviality}
	\mathcal{P}|_{J\times 0}\cong \check{N}_{J\times 0}\cong K_{J\times 0}\,,\quad  \mathcal{P}(J\times 0)|_{J\times 0}\cong \mathcal{O}_{J\times 0}\,.
\end{equation}

\begin{lem}\label{lemH0of1storder}
	One has
\begin{equation}
	\label{eqnglobalframe}
	H^{0}(J\times 0, \mathcal{P}(J\times 0)|_{2(J\times 0)})=\mathrm{span}\,
	\big\{f_{0}\,, f_{-1}+{\partial_{z}\theta(z)\over \theta(z)}f_{0}\big\}\,,
\end{equation}
where 	$f_0, f_{-1}$ are certain trivializations (detailed below) of $\mathcal{P}|_{J\times 0}, \mathcal{P}(J\times 0)|_{J\times 0}$, respectively. 
\end{lem}
\begin{proof}

The long exact sequence in cohomology associated to \eqref{eqndecompositionsequence}  gives
\begin{equation*}\label{eqnlongexactsequencecohomologyPoincare1storder}
0\rightarrow 
H^{0}(J\times 0, \mathcal{P}|_{J\times 0})
\rightarrow 
H^{0}(J\times 0, \mathcal{P}(J\times 0)|_{2(J\times 0)})
\rightarrow 
H^{0}(J\times 0, \mathcal{P}(J\times 0)|_{J\times 0})
\rightarrow \cdots
\end{equation*}
From \eqref{eqn0thordertriviality}, one has 
\begin{equation}\label{eqndimensionconstraintonP21storderrestriction}
	H^{0}(J\times 0, \mathcal{P}|_{J\times 0})
	\subseteq H^{0}(J\times 0, \mathcal{P}(J\times 0)|_{2(J\times 0)})\,,\quad 
\mathrm{dim}\, H^{0}(J\times 0, \mathcal{P}(J\times 0)|_{2(J\times 0)})\leq 2\,.
\end{equation}
Recall from \cite[Theorem 5.8, Theorem 6.5]{Raina:1989} (see also  \cite[Section 3.1]{Zhou:2023GWgeneratingseries} in the current notation) that when $c\neq 0~\mathrm{mod}~\Lambda_{\tau}$, one has
\begin{equation}\label{eqnScassectionofPoincare}
 H^{0}( J\times \mathsf{Pic}^{(0)}(J), \mathcal{P}(J\times 0))=\mathrm{span}\,\{ S_{c}(z)\}\,.
\end{equation}
Under the restriction map $\mathcal{P}(J\times 0)\rightarrow \mathcal{P}(J\times 0)|_{2(J\times 0)}$, the image of $S_{c}(z)$ in $H^{0}(J\times 0, \mathcal{P}(J\times 0)|_{2(J\times 0)})$ is represented by the 
first two terms of its the Laurent expansion in $c$
\begin{equation}\label{eqnfirstordernbd}
c^{-1}+c^{0}{\partial_{z}\theta(z)\over \theta(z)}\,,
\end{equation}
where $c^{-1}, c^{0}$ are trivializations of $\mathcal{P}(J\times 0)|_{J\times 0}\cong \mathcal{O}_{J\times 0},\mathcal{P}|_{J\times 0}\cong K_{J\times 0}$ respectively. 
Setting $f_{-1}=c^{-1}, f_{0}=c^{0}$, the desired relation \eqref{eqnglobalframe} then follows from
\eqref{eqndimensionconstraintonP21storderrestriction}
and
\eqref{eqnfirstordernbd}.

\end{proof}

Note that the $c^{-1}$-term of $S_{c}(z)$ is given by the residue
along the map  $\mathcal{P}(J\times 0)\rightarrow \mathcal{P}(J\times 0)|_{J\times 0}$
and is thus independent of the choice of local defining equation for $J\times 0$ inside  $J\times \mathsf{Pic}^{(0)}(J)$.
The quasi-ellipticity of $\theta'/\theta$ given in \eqref{eqne1automorphy} and
\eqref{eqnfirstordernbd} yield the automorphy
behavior
\begin{equation}
\label{eqnautomorphicbehaviorofc0c-1}
\begin{pmatrix} 
f_{-1} &
 f_{0}
 \end{pmatrix}
\mapsto
\begin{pmatrix}
f_{-1} &
f_{0}
 \end{pmatrix}
\begin{pmatrix}
1 & 0\\
\pii &1
\end{pmatrix}\,.
\end{equation}
From this we see the non-splitness of the decomposition sequence
\eqref{eqndecompositionsequence}.
Furthermore, from   \eqref{eqnautomorphicbehaviorofc0c-1}
it follows that  the real-analytic  trivializations 
\begin{equation}\label{eqnrealanalyticframe}
     \widehat{f}_{-1}:=f_{-1}-\mathbold{A}(z)f_{0}\,,\quad 
      \widehat{f}_{0}:=f_{0}
\end{equation}
are almost-elliptic, with the sections in \eqref{eqnglobalframe}  now expressed as
\begin{equation}\label{eqnglobalframerealanalytic}
 \widehat{f}_{0}\,,\quad     \widehat{f}_{-1}+{\widehat{\theta}'\over \widehat{\theta}}(z) \cdot   \widehat{f}_{0}\,.
\end{equation}


For convenience, in what follows we
replace $\mathcal{P}(J\times 0)$ by $\mathcal{L}(0\times  \mathsf{Pic}^{(0)}(J) )$, with
\begin{equation}
\mathcal{L}:=\mathcal{O}_{J\times  \mathsf{Pic}^{(0)}(J) }\left([\mathrm{ker}\,m]-[0\times   \mathsf{Pic}^{(0)}(J)]\right)\,.
\end{equation}
Working with the section $\theta(z+c)\over \theta(z)$ of $\mathcal{L}(0\times  \mathsf{Pic}^{(0)}(J) )$, the frame 	\eqref{eqnglobalframe} then changes accordingly to
\begin{equation}
	\label{eqnLaurentpolynomialsforL}
c^{1}\,,\quad 
c^{0}+{c^{1}\partial_{z}\theta(z)\over \theta(z)}\,,
\end{equation}
where $c^{0}, c^{1}$ are trivializations of 
$
\mathcal{L}(0\times  \mathsf{Pic}^{(0)}(J) )|_{J\times 0}\cong 
\mathcal{O}_{J\times 0}(0\times 0)
,
 \mathcal{L}(0\times  \mathsf{Pic}^{(0)}(J) )|_{J\times 0}\otimes \check{N}|_{J\times 0}
$,
  respectively. 

 In what follows we also make the following identifications without explicit mentioning 
\begin{equation}\label{eqntrivializationsofbundlesoverJtimes0}
J\times 0\cong J\,,\quad 
 \mathcal{O}_{J\times 0}\cong \mathcal{O}_{J}\cong K_{J}\,,\quad  \check{N}|_{J\times 0}\cong K_{J}\,.
 \end{equation}
 
We now consider the map
\begin{equation}
s_{ij}: 
E_{i}\times E_{\bar{j}}\rightarrow J_{i\bar{j}}\times E_{\bar{j}}\,,\quad 
 (w_{i},z_{j})\mapsto (w_{i}+z_{i}-z_{j},z_{j})\,,
 \end{equation}
where $J_{i\bar{j}}$ is a copy of $J$ labeled by $i,\bar{j}$.
We then 
apply the same discussions as above, with $J\times \mathrm{Pic}^{(0)}(J)$ replaced by $J_{i\bar{j}}\times E_{\bar{j}}$.
Then we have
\[
\prod_{i=1}^{n}
 {\theta(w_{i}+z_{i}-z_{j}+c)\over \theta(w_{i}+z_{i}-z_{j})}\in 
 H^{0}
\left(E_{[n]}\times E_{\bar{j}},  \otimes_{i=1}^{n} s_{ij}^{*}\mathcal{L} (0\times  \mathsf{Pic}^{(0)}(J) )\right)
\]
and thus
\begin{equation}\label{eqnsjexpansion}
s_{j}:=
f_{0}^{(j)}+f_{1}^{(j)}\sum_{i=1}^{n}{\theta'(w_{i}+z_{i}-z_{j})\over \theta(w_{i}+z_{i}-z_{j})}\in 
 H^{0}
\left(E_{[n]}\times E_{\bar{j}},  \otimes_{i=1}^{n} s_{ij}^{*}\mathcal{L}(0\times  \mathsf{Pic}^{(0)}(J) )|_{2(J\times 0)} \right)\,,
\end{equation}
where
$f_{0}^{(j)}$, 
 $f_{1}^{(j)}$ are the trivializations of $\mathcal{L}(0\times  \mathsf{Pic}^{(0)}(J) )|_{J_{i\bar{j}}\times 0},
  \mathcal{L}(0\times  \mathsf{Pic}^{(0)}(J) )|_{J_{i\bar{j}}\times 0}\otimes K_{J_{i\bar{j}}}$, respectively.

Denote 
 the trivializations  $(f_{0}^{(j)})_{1\leq j\leq n}, (f_{1}^{(j)})_{1\leq j\leq n}$ by $\mathbold{f}_{0}, \mathbold{f}_{1}$ collectively,  and similarly the frame
 $(s_{1},\cdots, s_{j},\cdots, s_{n})$ by  $\mathbold{s}$.
 Then \eqref{eqnsjexpansion} can be written  compactly as
 \begin{equation}\label{eqnmatrixpresentationofsjs}
\mathbold{s}=
\begin{pmatrix}
	\mathbold{f}_{0}&\mathbold{f}_{1}
\end{pmatrix}
\cdot
 \begin{pmatrix}
	\mathbb{1}\\
	Z_{n} 
\end{pmatrix}\,.
 \end{equation}
 Passing to the real-analytic trivializations  $\widehat{f}_{0}^{(j)},\widehat{f}_{1}^{(j)},j\geq 1$ constructed similar to that  in \eqref{eqnglobalframerealanalytic}, this gives 
  \begin{equation}\label{eqnmatrixpresentationofsjsrealanalytic}
\mathbold{s}=
\begin{pmatrix}
	\widehat{\mathbold{f}}_{0}& \widehat{\mathbold{f}}_{1}
\end{pmatrix}
\cdot
 \begin{pmatrix}
	\mathbb{1}\\
	\widehat{Z}_{n} 
\end{pmatrix}\,.
 \end{equation}
This can be regarded as describing a point in the Grassmannian
 \begin{equation}\label{eqnGrassmannian}
 \mathrm{Gr}\left(n,\, V_{n}\right)\,,\quad V_{n}:=\oplus_{j=1}^{n} \otimes_{i=1}^{n}s_{ij}^{*} H^{0}\left(J\times 0, \mathcal{L}(0\times  \mathsf{Pic}^{(0)}(J) )|_{2(J\times 0)} \right)\,,
 \end{equation}
 with $\det \widehat{Z}_{n} $ being one of the affine coordinates in the corresponding affine patch.
 To be more precise, using the frame $\mathbold{s}=(s_{j})_{1\leq j\leq n}$, the Pl\"ucker embedding is given by
\begin{eqnarray}
  \mathrm{Gr}\left(n,\, V_{n}\right)&\rightarrow& \mathbb{P} (\wedge^{n}V_{n})\,\nonumber\\
    \mathrm{span}\,(s_{j})_{1\leq j\leq n}&\mapsto&
[\wedge^{n}\mathbold{s}]\,, \quad\quad  \wedge^{n}\bold{s}=s_{1}\wedge s_{2}\wedge\cdots \wedge s_{n}\,.
 \end{eqnarray}
From \eqref{eqnmatrixpresentationofsjsrealanalytic}, we obtain the Pl\"ucker coordinates of $\wedge^{n}\mathbold{s}$:
\begin{equation}\label{eqnpluckercoordinates}
\wedge^{n}\mathbold{s}=\sum_{\emptyset \subseteq I\subseteq [n]} \wedge^{\mathrm{top}}\widehat{\mathbold{f}}_{I}\cdot \mathrm{det}\,\widehat{Z}_{I}\,,
\quad 
\wedge^{\mathrm{top}}\widehat{\mathbold{f}}_{I}:=\wedge_{j\in I^{c}}\widehat{f}_{0}^{(j)}\,\wedge \, \wedge_{j\in I}\widehat{f}_{1}^{(j)}\,.
\end{equation}
Introducing a pairing 
 $\langle-,-\rangle$ on $\wedge^{n}V_{n}$ by declaring $(\widehat{\mathbold{f}}_{I})_{I}$ to be an orthonormal basis, then 
 \[
 \langle \wedge^{\mathrm{top}}\widehat{\mathbold{f}}_{I}\, , \wedge^{n}\mathbold{s}\rangle=\mathrm{det}\,\widehat{Z}_{I}\,.
 \]

Due to
 \eqref{eqntrivializationsofbundlesoverJtimes0}, 
 we can identify $\mathbold{s}$ with
\begin{equation}\label{eqndeterminantlinebundlesectionasdeformation}
\mathbold{s}(\varepsilon):=
 \begin{pmatrix}
 	\widehat{\mathbold{f}}_{0}& \widehat{\mathbold{f}}_{1}
 \end{pmatrix}
 \cdot
  \begin{pmatrix}
 	\mathrm{vol} &0\\
 	0& \delta_{n}
 \end{pmatrix}
\cdot
 \begin{pmatrix}
 	\mathbb{1}\\
 	\widehat{Z}_{n} 
 \end{pmatrix}
\,, \quad \mathrm{vol}=
\begin{pmatrix}
 \mathrm{vol}_{1} &  \cdots &0\\
 \vdots	&  \mathrm{vol}_{k}& \vdots\\
	0 &\cdots& \mathrm{vol}_{n}\\
\end{pmatrix}\,.
\end{equation}
 It follows that
\begin{equation}\label{eqnPluckercoordinatesasdeformation}
\wedge^{n} \mathbold{s}(\varepsilon)=	 \sum_{\emptyset \subseteq I\subseteq [n]} \wedge^{\mathrm{top}}\widehat{\mathbold{f}}_{I}\cdot \varepsilon_{I}\cdot\mathrm{det}\,\widehat{Z}_{I}\cdot \bigwedge_{k=1}^{n}\mathrm{vol}_{k}\,.
\end{equation}
The iterated regularized integral then gives
\begin{equation}\label{eqnregularizedintegraloftopwedge}
\dashint_{E_{[n]}}\wedge^{n}\mathbold{s}(\varepsilon)
=	
\sum_{\emptyset \subseteq I\subseteq [n]}\wedge^{\mathrm{top}}\widehat{\mathbold{f}}_{I}\cdot \varepsilon_{I} \widehat{G}_{I}\,.
\end{equation}
In particular, we have
\[
\dashint_{E_{[n]}} \langle \wedge^{\mathrm{top}}
\widehat{\mathbold{f}}_{ [n]}, \wedge^{n} \mathbold{s}(\varepsilon)\rangle =\varepsilon_{[n]}\cdot\widehat{G}_{[n]}\,. 
\]

The relation \eqref{eqnregularizedintegraloftopwedge}  tells that the generating series $\widehat{\mathcal{G}}_{[n]}(\varepsilon)$ given in Definition \ref{dfnofgeneratingseriesofTn} can be regarded as the iterated regularized integral of $\wedge_{j=1}^{n}s_{j}$, which is a section 
 of the determinant line bundle over the Grassmannian $\mathrm{Gr}(n,V_{n})$.
 It is the truncation, induced by the map $\mathcal{L}|_{\infty(J\times 0)}\rightarrow  \mathcal{L}|_{2(J\times 0)}$, of the corresponding construction using the formal neighborhood $\infty(J\times 0)$.
Concretely, the latter is realized by 
 replacing the Taylor polynomial 	\eqref{eqnLaurentpolynomialsforL} by the full Taylor series.
 
The construction using $\mathcal{L}|_{2(J\times 0)}$ is also the deformation of the construction with $\varepsilon=0$, which corresponds to
 the restriction of $\mathcal{L}$ to the $0$th order neighborhood $\mathcal{L}|_{(J\times 0)}$
 instead of the restriction to the 1st order neighborhood $\mathcal{L}|_{2(J\times 0)}$.

\bibliographystyle{amsalpha}

\providecommand{\bysame}{\leavevmode\hbox to3em{\hrulefill}\thinspace}
\providecommand{\MR}{\relax\ifhmode\unskip\space\fi MR }
\providecommand{\MRhref}[2]{%
  \href{http://www.ams.org/mathscinet-getitem?mr=#1}{#2}
}
\providecommand{\href}[2]{#2}

\bigskip{}

\noindent{\small Yau Mathematical Sciences Center, Tsinghua University, Beijing 100084, P. R. China}

\noindent{\small Email: \tt jzhou2018@mail.tsinghua.edu.cn}

\end{document}